\documentclass[11pt]{amsart}

\usepackage{amsmath,amssymb,amscd,amsthm,graphicx,enumerate}

\usepackage{hyperref}
\usepackage{xypic}
\usepackage{xcolor}
\usepackage{color}
\hypersetup{breaklinks=true}
\hypersetup{
    colorlinks=true,
    linkcolor=red,
    filecolor=magenta,      
    urlcolor=cyan,
    citecolor=green,
}

\numberwithin{equation}{section}
\theoremstyle{plain}

\newtheorem{theorem}{Theorem}[section]
\newtheorem{proposition}[theorem]{Proposition}
\newtheorem{lemma}[theorem]{Lemma}

\newtheorem{claim}[theorem]{Claim}

\theoremstyle{remark}

\theoremstyle{definition}

\newcommand{\eps}{\varepsilon}

\begin{document}

\title{ Hearing the shape of ancient noncollapsed flows in $\mathbb{R}^{4}$}

\author{Wenkui Du, Robert Haslhofer}

\begin{abstract}
 We consider ancient noncollapsed mean curvature flows in $\mathbb{R}^4$ whose tangent flow at $-\infty$ is a bubble-sheet. We carry out a fine spectral analysis for the bubble-sheet function $u$ that measures the deviation of the renormalized flow from the round cylinder $\mathbb{R}^2 \times S^1(\sqrt{2})$ and prove that for $\tau\to -\infty$ we have the fine asymptotics $u(y,\theta,\tau)= (y^\top Qy -2\textrm{tr}(Q))/|\tau| + o(|\tau|^{-1})$,
 where $Q=Q(\tau)$ is a symmetric $2\times 2$-matrix whose eigenvalues are quantized to be either 0 or $-1/\sqrt{8}$. This naturally breaks up the classification problem for general ancient noncollapsed flows in $\mathbb{R}^4$ into three cases depending on the rank of $Q$. In the case $\mathrm{rk}(Q)=0$, generalizing a prior result of Choi, Hershkovits and the second author, we prove that the flow is either a round shrinking cylinder or $\mathbb{R}\times$2d-bowl. In the case $\mathrm{rk}(Q)=1$, under the additional assumption that the flow either splits off a line or is selfsimilarly translating, as a consequence of recent work by Angenent, Brendle, Choi, Daskalopoulos, Hershkovits, Sesum and the second author we show that the flow must be $\mathbb{R}\times$2d-oval or belongs to the one-parameter family of 3d oval-bowls constructed by Hoffman-Ilmanen-Martin-White, respectively. Finally, in the case $\mathrm{rk}(Q)=2$ we show that the flow is compact and $\mathrm{SO}(2)$-symmetric and for $\tau\to-\infty$ has the same sharp asymptotics as the $\mathrm{O}(2)\times\mathrm{O}(2)$-symmetric ancient ovals constructed by Hershkovits and the second author. The full classification problem will be addressed in subsequent papers based on the results of the present paper.
\end{abstract}

\maketitle

\tableofcontents

\section{Introduction}
In the analysis of mean curvature flow it is crucial to understand ancient noncollapsed flows. We recall that a mean curvature flow $M_t$ is called ancient if it is defined for all $t\ll 0$, and noncollapsed if it is mean-convex and there is an $\alpha>0$ so that every point $p\in M_t$ admits interior and exterior balls of radius at least $\alpha/H(p)$, c.f. \cite{ShengWang,Andrews_noncollapsing,HaslhoferKleiner_meanconvex} (in fact, by \cite{Brendle_inscribed,HK_inscribed} one can always take $\alpha=1$). By the work of White \cite{White_size,White_nature,White_subsequent} and by \cite{HH_subsequent} it is known that all blowup limits of mean-convex mean curvature flow are ancient noncollapsed flows. More generally, by Ilmanen's mean-convex neighborhood conjecture \cite{Ilmanen_problems}, which has been proved recently in the case of neck-singularities in \cite{CHH,CHHW}, it is expected even without mean-convexity assumption that all  blowup limits near any cylindrical singularity are ancient noncollapsed flows.

In a recent breakthrough \cite{BC1,BC2,ADS1,ADS2}, Brendle-Choi  and Angenent-Daskalopoulos-Sesum  classified all ancient noncollapsed flows in $\mathbb{R}^3$, or more generally in $\mathbb{R}^{n+1}$ under the additional assumption that the flow is uniformly two-convex. Specifically, they showed that any such flow is either a flat plane, a round shrinking sphere, a round shrinking cylinder, a translating bowl soliton, or an ancient oval. On the other hand, the classification of ancient noncollapsed flows in higher  dimensions without two-convexity assumption has remained a widely open problem.\footnote{There is a parallel story for the Ricci flow. Namely, a similar classification for ancient $\kappa$-noncollapsed flows in 3d Ricci flow, as conjectured by Perelman, has been obtained in \cite{Brendle_Bryant,ABDS,BDS}, with an extension to higher dimensions under the additional PIC2 assumption in \cite{LiZhang,BN,BDNS}, but the classification of general ancient $\kappa$-noncollapsed Ricci flows in higher dimensions has remained widely open.}

Very recently, Choi, Hershkovits and the second author made significant progress towards the classification of ancient noncollapsed flows in $\mathbb{R}^4$. Specifically, it has been shown in \cite{CHH_wing} that there do not exist any wing-like ancient noncollapsed flows, and it has been shown in \cite{CHH_translator} that any selfsimilarly translating solution is either $\mathbb{R}\times$2d-bowl, or a 3d round bowl, or belongs to the one-parameter family of $\mathbb{Z}_2\times\mathrm{O}(2)$-symmetric 3d oval-bowls constructed by Hoffman-Ilmanen-Martin-White \cite{HIMW}. The starting point for this analysis was to prove that for any noncompact ancient noncollapsed flow in  $\mathbb{R}^4$ that does not split off a line, the blowdown of any time-slice,
\begin{equation}
\check{M}_{t_0}:=\lim_{\lambda \to 0} \lambda M_{t_0},
\end{equation}
is always a ray. However, ancient noncollapsed flows can of course be either compact or noncompact. If the flow is compact, then the blowdown of any time-slice is just a point, and thus the blowdown of time-slices only in by itself does not provide sufficient information to get the analysis started.

In the present paper, we introduce a more general classification program for ancient noncollapsed flows in $\mathbb{R}^4$ that covers all cases, including the compact case. In fact, we will see that the classification problem actually involves three cases according to the rank of 
the so-called bubble-sheet matrix $Q$.

\bigskip

\subsection{Main results}
Let $M_t$ be an ancient noncollapsed flow in $\mathbb{R}^4$. To begin with, by the general theory from \cite{CM_uniqueness,HaslhoferKleiner_meanconvex}  the tangent  flow at $-\infty$ in suitable coordinates is always given by
\begin{equation}
\lim_{\lambda \rightarrow 0} \lambda M_{\lambda^{-2}t}=\mathbb{R}^{j}\times S^{3-j}(\sqrt{2(3-j)|t|})
\end{equation}
for some integer $0\leq j\leq 3$. If $j=3$ then the flow $M_t$ is simply a flat plane, and if $j=0$ then $M_t$ is simply a round shrinking sphere. If $j=1$ then by the classification of Brendle-Choi \cite{BC1,BC2} and Angenent-Daskalopoulos-Sesum \cite{ADS1,ADS2} the flow $M_t$ is a round shrinking cylinder, a round translating bowl soliton, or a two-convex ancient oval. Having discussed the known cases $j=0,1,3$ we can thus assume from now on that we are in general case $j=2$, namely that the tangent flow at $-\infty$ is a bubble-sheet:
\begin{equation}\label{bubble-sheet_tangent_intro}
\lim_{\lambda \rightarrow 0} \lambda M_{\lambda^{-2}t}=\mathbb{R}^{2}\times S^{1}(\sqrt{2|t|}).
\end{equation}
Equivalently, this means that the renormalized mean curvature flow
\begin{equation}
\bar M_\tau = e^{\frac{\tau}{2}}  M_{-e^{-\tau}}
\end{equation}
for $\tau\to -\infty$ converges to the normalized bubble-sheet
\begin{equation}
\Gamma=\mathbb{R}^2\times S^{1}(\sqrt{2}).
\end{equation}
Hence, we can write $\bar{M}_\tau$ as a graph of a function $u(\cdot,\tau)$ over $\Gamma\cap B_{\rho(\tau)}$, where $\rho(\tau)\to \infty$ as $\tau\to -\infty$, namely
\begin{equation}\label{graph_over_cylinder}
\left\{ q+ u(q,\tau)\nu(q) \, : \, q\in \Gamma\cap B_{\rho(\tau)} \right\} \subset \bar M_\tau\, ,
\end{equation}
where $\nu$ denotes the outwards unit normal of $\Gamma$. Our first main theorem describes the asymptotic behaviour of the bubble-sheet function $u$:

\begin{theorem}[bubble-sheet quantization]\label{spectral theorem}
For any ancient noncollapsed mean curvature flow in $\mathbb{R}^{4}$ whose tangent flow at $-\infty$ is given by \eqref{bubble-sheet_tangent_intro}, the bubble-sheet function $u$ satisfies
 \begin{equation}\label{main_thm_ancient}
\lim_{\tau\to -\infty} \Big\|\,
|\tau| u(y,\theta,\tau)- y^\top Qy +2\mathrm{tr}(Q)\, \Big\|_{C^{k}(B_R)} = 0
\end{equation}
for all $R<\infty$ and all integers $k$, where $Q$ is a symmetric $2\times 2$-matrix whose eigenvalues are quantized to be either 0 or $-1/\sqrt{8}$. Here, for $\mathrm{rk}(Q)\neq 1$ the matrix $Q$ is independent of time, while in the case $\mathrm{rk}(Q)=1$ we have
\begin{equation}
Q=R(\tau)^\top \begin{pmatrix}
0 & 0\\
0 & -1/\sqrt{8}
\end{pmatrix}R(\tau)
\end{equation}
for some rotation matrix $R(\tau)\in \mathrm{SO}(2)$ with $|\dot{R}(\tau)|=o(|\tau|^{-1})$.
\end{theorem}

This theorem, which we will prove via spectral analysis, shows that for $\tau\to -\infty$ the shape of $\bar{M}_\tau$ behaves in a highly specific way. Namely, to any ancient noncollapsed flow we can uniquely associate a symmetric $2\times 2$-matrix $Q=Q(\tau)$ whose eigenvalues are either $0$ or $-1/\sqrt{8}$ so that
 \begin{equation}\label{main_thm_ancient}
u(y,\theta,\tau)= \frac{y^\top Qy -2\textrm{tr}(Q)}{|\tau|} + o(|\tau|^{-1}).
\end{equation}
In particular, the flow becomes asymptotically $\mathrm{SO}(2)$-symmetric, and the hypersurfaces have an inwards quadratic bending shape over the $\mathbb{R}^2$-factor, where the bending coefficients are quantized to be either $0$ or $-1/\sqrt{8}$.\\

We remark that forwards in time related quantization behaviour, of course with the opposite sign, has been observed by Filippas-Liu for singularities of multidimensional semilinear heat equations \cite{FilippasLiu} and by Gang Zhou \cite{Zhou_dynamics} for certain cylindrical singularities under mean curvature flow.\\

According to Theorem \ref{spectral theorem} (bubble-sheet quantization), the problem of classifying general ancient noncollapsed flows in $\mathbb{R}^4$ naturally can be divided into the following three cases:
\begin{itemize}
\item the fully-degenerate case $\textrm{rk}(Q)=0$
\item the half-degenerate case $\textrm{rk}(Q)=1$
\item the non-degenerate case $\textrm{rk}(Q)=2$
\end{itemize}
We will now discuss these three cases in turn. Loosely speaking, the intuition is that directions in the range of $Q$ are short directions with inwards quadratic bending, while directions in the kernel of $Q$ are long directions.\\

In the fully-degenerate case $\textrm{rk}(Q)=0$ we prove:

\begin{theorem}[fully-degenerate case]\label{Rk=0}
If $\mathrm{rk}(Q)=0$, then $M_t$ is either a round shrinking $\mathbb{R}^2\times S^1$ or a translating $\mathbb{R}\times$2d-bowl. 
\end{theorem}

This generalizes a prior result by Choi, Hershkovits and the second author \cite{CHH_wing}, where it has been shown that any noncompact ancient noncollapsed flows in $\mathbb{R}^4$ whose tangent flow at $-\infty$ is given by \eqref{bubble-sheet_tangent_intro} and for which in the Merle-Zaag alternative, c.f. \cite{MZ}, the unstable mode is dominant, must be $\mathbb{R}\times$2d-bowl. Here, we rule out the potential scenario that there is any compact ancient noncollapsed flow in $\mathbb{R}^4$ with $\textrm{rk}(Q)=0$.\\

Next, in the half-degenerate case $\textrm{rk}(Q)=1$ we have:

\begin{theorem}[half-degenerate case]\label{thm_halfdeg}
If $\mathrm{rk}(Q)=1$, and if $M_t$ either splits off a line or is selfsimilarly translating, then $M_t$ is either $\mathbb{R}\times$2d-oval or belongs to the one-parameter family of 3d oval-bowls constructed by Hoffman-Ilmanen-Martin-White, respectively.
\end{theorem}

In fact, this is a rather direct consequence of of the recent work of Angenent, Brendle, Choi, Daskalopoulos, Hershkovits, Sesum and the second author \cite{ADS1,ADS2,BC1,CHH_wing,CHH_translator}. For convenience of the reader we include a proof explaining how this follows. The general case $\textrm{rk}(Q)=1$, without assuming selfsimilarity, will be addressed in forthcoming work by Choi, Hershkovits and the second author.\\

Finally, let us discuss the non-degenerate case $\textrm{rk}(Q)=2$. The simplest example in this case are the $\mathrm{O}(2)\times \mathrm{O}(2)$-symmetric ancient ovals constructed by Hershkovits and the second author \cite{HaslhoferHershkovits_ancient}. In our recent paper \cite{DH_ovals}, we proved that these solutions are unique  among $\mathrm{SO}(2)\times \mathrm{SO}(2)$-symmetric compact ancient noncollapsed flows, as conjectured by Angenent-Daskalopoulos-Sesum. On the other hand, in the same paper we
 also constructed a one-parameter family of ancient ovals that are only $\mathbb{Z}_2\times \mathrm{O}(2)$-symmetric. These ovals have long axes of different length, and interpolate between $\mathbb{R}\times$2d-oval and 2d-oval$\times\mathbb{R}$, and can be thought of as a compact version of the Hoffman-Ilmanen-Martin-White examples.

Here, we prove that any ancient noncollapsed flow in $\mathbb{R}^4$ with $\mathrm{rk}(Q)=2$ is compact and $\mathrm{SO}(2)$-symmetric and for $\tau\to -\infty$ has the same sharp asymptotics as the $\mathrm{SO}(2)\times \mathrm{SO}(2)$-symmetric ancient ovals from \cite{HaslhoferHershkovits_ancient}:

\begin{theorem}[non-degenerate case]\label{Rk=2}
If $\mathrm{rk}(Q)=2$, then $M_t$ is compact and $\mathrm{SO}(2)$-symmetric and satisfies the following sharp asymptotics:
\begin{itemize}
    \item Parabolic region: The bubble-sheet function $u$ for $\tau\to -\infty$ satisfies
    \begin{equation*}
     u(y_1,y_2,\theta,\tau)=\frac{y_1^2+y_2^2-4}{\sqrt{8}\tau} +o(|\tau|^{-1})
    \end{equation*}
        uniformly for $|(y_1,y_2)|\leq R$.
      \item Intermediate region: We have
\begin{equation*}
    \lim_{\tau\rightarrow -\infty}u(|\tau|^{\frac{1}{2}}z_1,|\tau|^{\frac{1}{2}}z_2,\theta, \tau)+\sqrt{2}=\sqrt{2-(z_1^2+z_2^2)}
\end{equation*}
    uniformly on every compact subset of $\{ z_1^2+z_2^2<\sqrt{2}\}$.
    \item Tip region:  Setting $\lambda(s)=\sqrt{|s|^{-1}\log |s|}$, and given any angle $\phi$ letting $p_{s}\in M_s$ be the point that maximizes $\langle p, \cos(\phi) e_1 + \sin(\phi) e_2\rangle$ among all $p\in M_s$, as $s\to -\infty$ the rescaled flows 
    \begin{equation*}
        {\widetilde{M}}^{s}_{t}=\lambda(s)\cdot(M_{s+\lambda(s)^{-2}t}-p_{s})
    \end{equation*}
    converge to $\mathbb{R}\times N_{t}$, where $N_{t}$ is the 2d-bowl in $\mathbb{R}^{3}$ with speed $1/\sqrt{2}$.
\end{itemize}
\end{theorem}

For necks and certain bubble-sheets related sharp asymptotics have been upgraded to classification results in \cite{ADS2,DH_ovals,CHH_translator}. In light of these prior results it seems likely that our sharp asymptotics from Theorem \ref{Rk=2} (non-degenerate case) can be upgraded to show that any ancient noncollapsed flow in $\mathbb{R}^4$ with $\mathrm{rk}(Q)=2$ is either the $\mathrm{O}(2)\times \mathrm{O}(2)$-symmetric ancient oval from \cite{HaslhoferHershkovits_ancient} or belongs to the one-parameter family of $\mathbb{Z}_2\times \mathrm{O}(2)$-symmetric ovals from \cite{DH_ovals}. We will address this in subsequent work.

\bigskip

\subsection{Outline of the proofs}
Let us now outline the main ideas.

In Section \ref{sec_fine_bubble_sheet}, we set up the fine bubble-sheet analysis. This is essentially a bubble-sheet version of what has been done for necks in \cite[Section 4]{CHH}. Specifically, since $\bar{M}_\tau$ moves by renormalized mean curvature flow, the evolution of the bubble-sheet function $u(\cdot,\tau)$ over $\Gamma=\mathbb{R}^2\times S^1(\sqrt{2})$, as defined in \eqref{graph_over_cylinder}, is governed by the Ornstein-Uhlenbeck type operator
\begin{equation}
\mathcal L=\frac{\partial^2}{\partial y_1^2}-\frac{y_1}{2}\frac{\partial}{\partial y_{1}}
+\frac{\partial^2}{\partial y_2^2}-\frac{y_2}{2}\frac{\partial}{\partial y_{2}}
+\frac{1}{2}\frac{\partial^2}{\partial \theta^2}+1\, .
\end{equation}
This operator has 5 unstable eigenfunctions, namely
\begin{align}
1, y_1, y_{2}, \cos\theta, \sin\theta\, ,\label{basis_hplus_intro}
\end{align}
and 7 neutral eigenfunctions, namely
\begin{align}
y^2_{1}-2, y^2_{2}-2, y_{1}y_{2}, y_{1}\cos\theta, y_{1}\sin\theta,  y_{2}\cos\theta, y_{2}\sin\theta \, ,
\label{basis_hneutral_intro}
\end{align}
and all the other eigenfunctions are stable. By the Merle-Zaag alternative from \cite{MZ}, for $\tau\to -\infty$ either the unstable eigenfunctions are dominant or the neutral eigenfunctions are dominant. In the unstable case it is not hard to see that the function $u$ decays exponentially, and thus the spectral quantization theorem holds with $Q=0$. Hence, we can focus on the case where the neutral eigenfunctions are dominant. Using the Lojasiewicz inequality from Colding-Minicozzi \cite{CM_uniqueness} and barrier arguments, we then show that $\rho(\tau)=|\tau|^{\gamma}$, for some $\gamma>0$, is an admissible graphical radius, i.e.
\begin{equation}\label{small_graph_intro}
\|u(\cdot,\tau)\|_{C^4(\Gamma \cap B_{2\rho(\tau)}(0))} \leq  \rho(\tau)^{-2}.
\end{equation}
Moreover, using Zhu's bubble-sheet  improvement theorem from \cite{Zhu} we show that the central region of the hypersurfaces is almost $\mathrm{SO}(2)$-symmetric,  specifically that there is some $\eta>0$ such that for all $\tau\ll 0$ we have
\begin{equation}\label{eq_alm_symm_intro}
  \sup_{|(y_1,y_2)|\leq \rho(\tau)}   |u_{\theta}(y_1,y_2,\theta,\tau)|\leq e^{-\eta \rho(\tau)} \, .
\end{equation}
This has the important consequence that out of the $7$ eigenfunctions listed in \eqref{basis_hneutral_intro} only the first 3 can be dominant. Moreover, it also implies that if we Taylor expand the evolution of the truncated bubble-sheet function
\begin{equation}
\hat{u}(y_1,y_2,\theta,\tau):=u(y_1,y_2,\theta,\tau) \chi\left(\frac{|(y_1,y_2)|}{\rho(\tau)}\right) \, ,
\end{equation}
where $\chi$ is a suitable cutoff function, then to second order we have
\begin{equation}\label{sec_taylor_intro}
 \partial_{\tau}\hat{u}=\mathcal{L}\hat{u}-\frac{1}{\sqrt{8}}\hat{u}^2+\ldots\, ,
\end{equation}
up to controllable error terms.\\

In Section \ref{sec_quant_thm}, we prove Theorem \ref{spectral theorem} (bubble-sheet quantization). To this end, we consider the expansion
\begin{equation}\label{expansion_123}
 \hat{u} = \alpha_1(y^2_{1}-2)+\alpha_2(y^2_{2}-2)+ 2\alpha_3y_{1}y_{2}+w\, ,
 \end{equation}
 where the remainder term $w$ is controllable thanks to the assumption that the neutral eigenfunctions from \eqref{basis_hneutral_intro} are dominant and thanks to the almost circular symmetry from \eqref{eq_alm_symm_intro}.
Taking also into account \eqref{sec_taylor_intro} we then show that the spectral coefficients $\vec{\alpha}=(\alpha_1,\alpha_2,\alpha_3)$ evolve by
\begin{equation}\label{odes0_intro}
 \begin{cases}
   \dot{\alpha}_{1}=-\sqrt{8}(\alpha^2_{1}+\alpha_{3}^2)+o(|\vec{\alpha}|^2+|\tau|^{-100})\\
   \dot{\alpha}_{2}=-\sqrt{8}(\alpha^2_{2}+\alpha_{3}^2)+o(|\vec{\alpha}|^2+|\tau|^{-100})\\
    \dot{\alpha}_{3}=-\sqrt{8}(\alpha_{1}+\alpha_2)\alpha_{3}+o(|\vec{\alpha}|^2+|\tau|^{-100})\, .
    \end{cases}
\end{equation}
These spectral ODEs are rather tricky to analyze. To get a hold on them, we start by establishing the  a priori estimates
\begin{equation}
\alpha_1 \leq o(|\vec{\alpha}|+|\tau|^{-100})\, , \quad \alpha_2 \leq o(|\vec{\alpha}|+|\tau|^{-100})\, ,
\end{equation}
and
\begin{equation}
\alpha_3^2 - \alpha_1\alpha_2 \leq o(|\vec{\alpha}|^2+|\tau|^{-100})\, .
\end{equation}
In essence, these estimates come from convexity, similarly as in \cite{CHH_wing}, though some care is needed to handle the rotations and error terms.
We then consider the trace and determinant,
\begin{equation}
    S:=\alpha_{1}+\alpha_{2}\, ,\quad D:=\alpha_{1} \alpha_{2} -{\alpha_{3}^{2}} \, ,
\end{equation}
and show that they evolve by 
\begin{equation}\label{Riccati ODEs_intro}
    \begin{cases}
      \dot{S}=-\sqrt{8}(S^2-2D)+o(|\tau|^{-2})\, ,\\
       \dot{D}=-\sqrt{8}SD+o(|\tau|^{-3})\, .
    \end{cases}
    \end{equation}
   Moreover, using in particular our a priori estimates we prove that
     \begin{equation}
   \frac{1+o(1)}{\sqrt{2}\tau} \leq S\leq  \frac{1-o(1)}{\sqrt{8}\tau}\, , \qquad 
 -\frac{o(1)}{\tau^2} \leq D \leq  \frac{1+o(1)}{8\tau^2}\, .
    \end{equation}
    Then, carefully analyzing the Riccati type ODEs from \eqref{Riccati ODEs_intro} we prove the crucial result that actually only the extremal cases can occur. Namely, we show that the trace and determinant satisfy either
\begin{equation}\label{case1_intro}
S=\frac{1}{\sqrt{8}\tau}+o(|\tau|^{-1})\, , \quad
D=o(|\tau|^{-2})\, ,
\end{equation}
or
    \begin{equation}
\label{case2_intro}
S=\frac{1}{\sqrt{2}\tau}+o(|\tau|^{-1})\, ,\quad 
D=\frac{1}{8\tau^2}+o(|\tau|^{-2})\, .
\end{equation}
To convey the intuition behind this dichotomy, let us point out that scenario \eqref{case1_intro} corresponds to the case $\mathrm{rk}(Q)=1$ and occurs for example for the 3d oval-bowls, while scenario \eqref{case2_intro} corresponds to the case $\mathrm{rk}(Q)=2$ and occurs for example for the $\mathrm{O}(2)\times\mathrm{O}(2)$-symmetric ancient ovals. Finally, having established the behavior of the trace and determinant, we can eventually solve the original spectral ODEs \eqref{odes0_intro}, and prove that the conclusion
 \begin{equation}\label{main_thm_ancient_rest}
\lim_{\tau\to -\infty} \Big\|\,
|\tau| u(y,\theta,\tau)- y^\top Qy +2\mathrm{tr}(Q)\, \Big\|_{C^{k}(B_R)} = 0
\end{equation}
 holds with
\begin{equation}\label{eq_mat_q}
Q=
R(\tau)^\top \begin{pmatrix}
0 & 0\\
0 & -1/\sqrt{8}
\end{pmatrix} R(\tau),
\end{equation}
for some $R(\tau)\in \mathrm{SO}(2)$ satisfying $|\dot{R}(\tau)|=o(|\tau|^{-1})$,
respectively with
\begin{equation}\label{eq_mat_q}
Q=
\begin{pmatrix}
-1/\sqrt{8} & 0\\
0 & -1/\sqrt{8}\ 
\end{pmatrix}.
\end{equation}
This completes our outline of the proof of the bubble-sheet quantization theorem.\\

In Section \ref{sec_full_deg}, we prove Theorem \ref{Rk=0} (fully-degenerate case). As discussed, it suffices to prove that any solution with $\mathrm{rk}(Q)=0$ must be noncompact. To show this, we use that by the fine bubble-sheet theorem from \cite{CHH_wing} there is a nonvanishing fine bubble-sheet vector $(a_{1}, a_{2})$ associated to our flow, such that for any space-time point $X$ after suitable recentering in the $x_3x_4$-plane the profile function $u^X$ of the renormalized flow $\bar{M}_\tau^X$ centered at $X$ satisfies
\begin{equation}
u^X = e^{\tau/2}(a_1y_1+a_2y_2) + o(e^{\tau/2})
\end{equation}
for all $\tau\ll 0$ depending only on the bubble-sheet scale. Supposing towards a contradiction that there is a compact example with $\mathrm{rk}(Q)=0$ we then blow up around suitable tip points that maximize respectively minimize the value of $x_1$, and argue that the two blowup limits must have fine bubble-sheet vectors pointing in opposing directions. This contradicts the fact that $(a_{1}, a_{2})$ is independent of the center point, and thus concludes the outline of the proof.\\

In Section \ref{sec_half_deg}, we prove Theorem \ref{thm_halfdeg} (half-degenerate case).  Specifically, in case the flow splits off a line, then using the classification from \cite{BC1,ADS1,ADS2} we show that it must be $\mathbb{R}\times$2d-oval, and in case the flow is selfsimilarly translating, then using the classification from \cite{CHH_translator} we show that it belongs to the one-parameter family of 3d oval-bowls.\\

Finally, in Section \ref{sec_non_deg}, we prove Theorem \ref{Rk=2} (non-degenerate case). We first establish the sharp asymptotics by generalizing the arguments from our prior paper \cite{DH_ovals} to the setting without symmetry assumptions. In particular, to establish the sharp asymptotics in the intermediate region we first establish the lower bound using interior barriers anchored in the central region, next use again Zhu's bubble-sheet improvement to prove almost symmetry up to renormalized radius $(\sqrt{2}-o(1))\sqrt{|\tau|}$, and then generalize the supersolution estimate from our prior paper to this setting with error terms to establish the upper bound. In particular, the sharp asymptotics imply that the hypersurfaces must be compact. Having established the sharp asymptotics we show that at all sufficiently negative times every point in our solution has a canonical neighborhood modelled either on a round shrinking $\mathbb{R}^2\times S^1$ or on a translating $\mathbb{R}\times$2d-bowl. Finally, similarly as in \cite{BC1,ADS2,Zhu} we establish $\mathrm{SO}(2)$-symmetry via bubble-sheet improvement and cap improvement. This concludes the outline of the proof.\\

\bigskip

\noindent\textbf{Acknowledgments.}
This research was supported by the NSERC Discovery Grant and the Sloan Research Fellowship of the second author.

\bigskip

\section{Fine bubble-sheet analysis and almost symmetry}\label{sec_fine_bubble_sheet}
In  this section, we set up the fine bubble-sheet analysis and show that the central region is $\textrm{SO}(2)$-symmetric up to exponential error terms. This is essentially a bubble-sheet version of what has been done for necks in \cite[Section 4]{CHH}. A related bubble-sheet analysis has been done in \cite{CHH_wing}. However, in contrast to the setup from \cite{CHH_wing}, where a fine-tuning rotation, similarly as in  \cite{BC1}, has been used to project away from certain rotations, for our present purpose it is better to instead use the Lojasiewicz inequality from \cite{CM_uniqueness}, similarly as in \cite{CHH}, to control the rotations.\\

Throughout this section, we will need some general facts about barriers, which we will recall now. By \cite[Section 4]{ADS1} there is some $L_0>1$ such that  for every $a\geq L_0$ there is a shrinker-with-boundary in $\mathbb{R}^3$,
\begin{align}\label{ads_shrinker}
{\Sigma}_a &= \{ \textrm{surface of revolution with profile } r=u_a(y_1), 0\leq y_1 \leq a\}\, .
\end{align}
Here, the parameter $a$ captures where the  concave functions $u_a$ meet the $y_1$-axis, namely $u_a(a)=0$. These shrinkers foliate the region enclosed by the cylinder $\mathbb{R}\times S^1(\sqrt{2})$ for $|y_1|\geq L_0$. In \cite[Section 3]{CHH_wing} the 2d ADS-shrinkers have been shifted and rotated to construct the 3d hypersurfaces
\begin{equation}\label{rotated_barrier}
\Gamma_a=\{(r\cos\theta ,r\sin\theta,y_3,y_4)\in  \mathbb{R}^4:\theta\in [0,2\pi), (r-1,y_3,y_4) \in {\Sigma}_a \}.
\end{equation}
By construction the hypersurfaces $\Gamma_a$ foliate the region enclosed by the cylinder $\Gamma=\mathbb{R}^2\times S^1(\sqrt{2})$ for $|(y_1,y_2)|\geq L_0+1$, and by  \cite[Corollary 3.4]{CHH_wing} they act as inner barriers for the renormalized mean curvature flow. 

\bigskip

\subsection{Basic bubble-sheet setup} Our basis setup is similar as in \cite[Section 4]{CHH_wing} with the difference that we now work without fine-tuning rotation. For convenience of the reader we give an essentially self-contained exposition.

Let $M_t$ be an ancient noncollapsed mean curvature flow in $\mathbb{R}^{4}$, whose tangent flow at $-\infty$ in suitable coordinates is
\begin{equation}\label{bubble-sheet_tangent}
\lim_{\lambda \rightarrow 0} \lambda M_{\lambda^{-2}t}=\mathbb{R}^{2}\times S^{1}(\sqrt{2|t|}).
\end{equation}
In other words, the renormalized mean curvature flow,
\begin{equation}
\bar M_\tau = e^{\frac{\tau}{2}}  M_{-e^{-\tau}},
\end{equation}
converges for $\tau\to -\infty$ to the bubble-sheet
\begin{equation}
\Gamma=\mathbb{R}^2\times S^{1}(\sqrt{2}).
\end{equation}
Assume further that $M_t$ is not a round shrinking cylinder.
Let us fix some admissible graphical radius function $\rho(\tau)$ for $\tau\ll 0$, namely a positive function satisfying
\begin{equation}\label{univ_fns}
\lim_{\tau \to -\infty} \rho(\tau)=\infty, \quad \textrm{and}\quad  -\rho(\tau) \leq \rho'(\tau) \leq 0,
\end{equation}
so that $\bar M_\tau$ can be written as the graph of a function $u(\cdot,\tau)$ over $\Gamma\cap B_{2\rho(\tau)}$ with the estimate
\begin{equation}\label{small_graph}
\|u(\cdot,\tau)\|_{C^4(\Gamma \cap B_{2\rho(\tau)}(0))} \leq  \rho(\tau)^{-2}.
\end{equation}
Since $\bar{M}_\tau$ moves by renormalized mean curvature flow, the bubble-sheet function $u$ evolves by
 \begin{equation}
      \partial_{\tau}u=\mathcal{L}u+E,
 \end{equation}
where $\mathcal L$ is an Ornstein-Uhlenbeck type operator on $\Gamma=\mathbb{R}^2\times S^{1}(\sqrt{2})$ explicitly given  by
\begin{equation}
\mathcal L=\frac{\partial^2}{\partial y_1^2}-\frac{y_1}{2}\frac{\partial}{\partial y_{1}}
+\frac{\partial^2}{\partial y_2^2}-\frac{y_2}{2}\frac{\partial}{\partial y_{2}}
+\frac{1}{2}\frac{\partial^2}{\partial \theta^2}+1,
\end{equation}
and where the error term thanks to \eqref{small_graph} satisfies the pointwise estimate
\begin{equation}\label{error_first_untrunc}
    |E|\leq C\rho^{-2}(|u|+|\nabla u|).
\end{equation}
Denote by $\mathcal{H}$ the Hilbert space of Gaussian $L^2$ functions on $\Gamma$, where
\begin{equation}\label{def_norm}
\langle f, g\rangle_{\mathcal{H}}= \frac{1}{(4\pi)^{3/2}} \int_{\Gamma}  f(q)g(q) e^{-\frac{|q|^2}{4}} \, dq\, .
\end{equation}
We also fix a nonnegative smooth cutoff function $\chi$ satisfying $\chi(s)=1$ for $|s| \leq \tfrac{1}{2}$ and $\chi(s)=0$ for $|s| \geq 1$, and consider the truncated function
\begin{equation}
\hat{u}(y_1,y_2,\theta,\tau):=u(y_1,y_2,\theta,\tau) \chi\left(\frac{|(y_1,y_2)|}{\rho(\tau)}\right).
\end{equation}
\begin{proposition}[{truncated evolution, cf. \cite[Proposition 4.6]{CHH_wing}}]\label{prop_trunc_bubble}
The truncated bubble-sheet function $\hat{u}$ satisfies
\begin{equation}\label{eq_truncated}
\left\|(\partial_\tau -\mathcal{L})\hat{u} \right\|_{\mathcal{H}}\leq C\rho^{-1} \| \hat{u} \|_{\mathcal{H}}.
\end{equation}
\end{proposition}
\begin{proof} Writing $r=|(y_1,y_2)|$ we compute
\begin{multline} 
(\partial_\tau -\mathcal{L})\hat{u}  = E \, \chi\! \Big ( \frac{r}{\rho} \Big ) - \frac{2}{\rho} \, \frac{\partial u}{\partial r} \, \chi' \Big ( \frac{r}{\rho} \Big ) - \frac{1}{\rho^2} \, u \, \chi'' \Big ( \frac{r}{\rho} \Big )\\
+ \frac{r}{2\rho} \, u \, \chi' \Big ( \frac{r}{\rho} \Big ) - \frac{r \rho'}{\rho^2} \, u \, \chi' \Big ( \frac{r}{\rho} \Big ). 
\end{multline} 
For $r \leq \rho/2$ using \eqref{error_first_untrunc} we infer that
\begin{equation}
|(\partial_\tau -\mathcal{L})\hat{u}| \leq \frac{C}{\rho^{2}}( |u| + |\nabla u|) \, ,
\end{equation}
while for $\rho/2\leq r \leq \rho$ using in addition \eqref{univ_fns} we see that
\begin{equation}
|(\partial_\tau -\mathcal{L})\hat{u}| \leq C |u| + \frac{C}{\rho}|\nabla u| \, .
\end{equation}
Together with the inverse Poincare inequality and the weighted $L^2$-estimate from \cite[Proposition 4.4]{CHH_wing}, which in particular hold in the case where the fine-tuning rotation is simply the identity matrix, this yields
\begin{align}
\int_\Gamma |(\partial_\tau -\mathcal{L})\hat{u}|^2 e^{-\frac{|q|^2}{4}}&\leq \frac{C}{\rho^{2}} \int_{\Gamma} (u^2+|\nabla u|^2) e^{-\frac{|q|^2}{4}} + C \int_{\Gamma\cap \{ \rho/2\leq r \leq \rho \} } u^2 e^{-\frac{|q|^2}{4}}\nonumber\\
&\leq \frac{C}{\rho^{2}} \int_{\Gamma\cap \{  r \leq \rho/2 \} } \hat{u}^2 e^{-\frac{|q|^2}{4}}\, ,
\end{align}
which proves the proposition. 
\end{proof}
Analyzing the spectrum of $\mathcal L$, we can decompose our Hilbert space as
\begin{equation}
\mathcal H = \mathcal{H}_+\oplus \mathcal{H}_0\oplus \mathcal{H}_-,
\end{equation}
where the unstable space is given by
\begin{align}
\mathcal{H}_+ =\text{span}\{1, y_1, y_{2}, \cos\theta, \sin\theta\},\label{basis_hplus}
\end{align}
and neutral space is given by
\begin{align}
\mathcal{H}_0 =\text{span}\left\{y^2_{1}-2, y^2_{2}-2, y_{1}y_{2}, y_{1}\cos\theta, y_{1}\sin\theta,  y_{2}\cos\theta, y_{2}\sin\theta \right\}.
\label{basis_hneutral}
\end{align}
Consider the functions
\begin{equation}
U_{\pm}(\tau):= \|P_{\pm} \hat{u}(\cdot,\tau)\|_{\mathcal{H}}^2, \qquad U_0(\tau):= \|P_0 \hat{u}(\cdot,\tau)\|_{\mathcal{H}}^2\, ,
\end{equation}
where $\mathcal{P}_{\pm},\mathcal{P}_0$ denotes the orthogonal projections to $\mathcal{H}_{\pm},\mathcal{H}_0$, respectively.
\begin{proposition}[{Merle-Zaag alternative, cf. \cite[Theorem 4.8]{CHH_wing}}]\label{mz.ode.fine.bubble-sheet}
For $\tau\to -\infty$, either the neutral mode is dominant, i.e.
\begin{equation}
U_-+U_+=o(U_0),
\end{equation}
or the unstable mode is dominant, i.e.
\begin{equation}
U_-+U_0\leq C\rho^{-1}U_+.
\end{equation}
\end{proposition}
\begin{proof}
Using Proposition \ref{prop_trunc_bubble} (truncated evolution) and observing that all nonzero eigenvalues of $\mathcal{L}$ have absolute value at least $1/2$, we obtain
\begin{align} 
 \dot{U}_+ &\geq U_+ - C\rho^{-1} \, (U_+ + U_0 + U_-), \nonumber\\ 
 | \dot{U}_0  | &\leq C\rho^{-1} \, (U_+ + U_0 + U_-), \label{U_PNM_system}\\ 
 \dot{U}_- &\leq -U_- + C\rho^{-1} \, (U_+ + U_0 + U_-). \nonumber
\end{align}
Moreover, since $\bar{M}_\tau$ for $\tau\to -\infty$ converges locally uniformly to $\Gamma$, we see that
\begin{equation}
\lim_{\tau\to -\infty} \left( U_{+} + U_0+U_{-}\right) =0\, .
\end{equation}
Hence, the Merle-Zaag ODE lemma \cite{MZ,ChoiMantoulidis} implies the assertion.
\end{proof}

To conclude this subsection, let us observe that for any other choice of admissible graphical radius the same mode stays dominant.

\bigskip

\subsection{Graphical radius}
Throughout this subsection we assume that the neutral mode is dominant. The goal is to construct an improved graphical radius, by generalizing \cite[Section 4.3.1]{CHH} to the  bubble-sheet setting. To this end, we denote by $\rho_0$ the initial choice of graphical radius from the previous subsection, and consider the quantities
\begin{equation}\label{def_beta}
    \beta(\tau):=\sup_{\sigma\leq \tau}\left(\int_{\Gamma}{u}^2(q,\sigma)\chi^2\left(\frac{|q|}{\rho_{0}(\sigma)}\right)e^{-\frac{|q|^2}{4}}dq\right)^{1/2},
\end{equation}
and
\begin{equation}\label{graphrho}
    \rho(\tau):=\beta(\tau)^{-\frac{1}{5}}.
\end{equation}

\begin{proposition}[{admissibility}]\label{improved radius}\label{prop_admiss}
The function $\rho$ is an admissible graphical radius function, i.e. the estimates \eqref{univ_fns} and \eqref{small_graph} hold for $\tau\ll 0$.
\end{proposition}

\begin{proof}
Since $\bar{M}_\tau$ for $\tau\to -\infty$ converges locally uniformly to $\Gamma$, it is clear that
\begin{equation}
\lim_{\tau \to -\infty} \rho(\tau)=\infty\, .
\end{equation}
To proceed, we need the following barrier estimate:

\begin{claim}[{barrier estimate, cf. \cite[Proposition 4.18]{CHH}}]
There is a constant $C<\infty$ such that
\begin{equation}
 |u (y_1,y_2,\theta,\tau)| \leq C\beta(\tau)^{\frac{1}{2}}
\end{equation}
holds for $|(y_1,y_2)| \leq C^{-1}\beta(\tau)^{-\frac{1}{4}}$ and $\tau\ll 0$.
\end{claim}

\begin{proof}Let $L_0$ be the constant from the ADS-shrinker foliation \eqref{ads_shrinker}.
By standard parabolic estimates there is some constant $K<\infty$ such that for $\tau\ll 0$ we have
\begin{equation}\label{C_0.bound.y=L_0}
\sup_{|(y_1,y_2)|\leq 2L_0}|u(y_1,y_2,\theta,\tau)|\leq K \beta(\tau).
\end{equation}
Now, given $\hat \tau\ll 0$, consider  the hypersurface $\Gamma_a$ from \eqref{rotated_barrier} with parameter
\begin{equation}\label{a=root.beta}
a=\frac{c_0}{\sqrt{K\beta(\hat \tau)}}.
\end{equation}
If we choose $c_0$ small enough, then by \cite[Lemma 4.4]{ADS1} the profile function $u_a$ of the ADS-shrinker $\Sigma_a$ satisfies
\begin{equation}
u_a(L_0-1) \leq \sqrt{2}-K\beta(\hat \tau).
\end{equation}
Combining this with \eqref{C_0.bound.y=L_0}, the inner barrier principle from  \cite[Corollary 3.4]{CHH_wing} implies that $\Gamma_a$ is enclosed by $\bar M_\tau$  for $|(y_1,y_2)| \geq L_0$ and $\tau \leq \hat\tau$. Since  $u_a(\sqrt{a})^2\geq 2-2/a$ (see e.g. \cite[Equation (195)]{CHH}), this yields
\begin{equation}
u(y_1,y_2,\theta,\hat\tau)^2 \geq 2-2/a
\end{equation}
for $|(y_1,y_2)|\in [L_0,\sqrt{a}-1]$. Hence, remembering \eqref{a=root.beta} we conclude that 
\begin{equation}
 u (y_1,y_2,\theta,\tau) \geq -C\beta(\tau)^{\frac{1}{2}}
\end{equation}
holds for $|(y_1,y_2)| \leq C^{-1}\beta(\tau)^{-\frac{1}{4}}$ and $\tau\ll 0$. Finally, by convexity and \eqref{C_0.bound.y=L_0} the lower bound implies a corresponding upper bound. This finishes the proof of the claim.
\end{proof}

Now, remembering that  $\rho(\tau)=\beta(\tau)^{-\frac{1}{5}}$, by the claim and standard interior estimates we get
\begin{equation}
\|u(\cdot,\tau)\|_{C^4(\Sigma\cap B_{2\rho(\tau)}(0))}\leq \rho(\tau)^{-2}
\end{equation}
for $\tau\ll 0$. Moreover, by definition of $\beta$ we clearly have $\dot{\rho}\leq 0$. Finally, using the assumption that the neutral eigenfunctions dominate and the weighted $L^2$-estimate from \cite[Proposition 4.4]{CHH_wing} we see that
\begin{equation}
\left|\frac{d}{d\tau} \beta^2\right| = o(\beta^2),
\end{equation}
which implies $ |\dot{\rho}|\leq \rho$ for $\tau \ll 0$. This finishes the proof of proposition.
\end{proof}

\begin{proposition}[{graphical radius, c.f.  \cite[Proposition 4.19]{CHH}}]\label{radius_lower_bound}
There is a constant $\gamma>0$, so that $\rho(\tau)=\beta(\tau)^{-\frac{1}{5}}$ for $\tau\ll 0$ satisfies
\begin{equation}
    \rho(\tau)\geq |\tau|^{\gamma}.
\end{equation}
\end{proposition}

\begin{proof}
Consider the Gaussian area functional
\begin{equation}
F(M)=(4\pi)^{-3/2}\int_M e^{-\frac{|q|^2}{4}} \, .
\end{equation}
By Colding-Minicozzi \cite[Theorem 6.1]{CM_uniqueness}, there exist $\eta\in (1/3,1)$ and $K<\infty$, so that for $\tau\ll 0$ we have
\begin{equation}
\left (F(\Gamma)-F(\bar{M}_{\tau})\right)^{1+\eta} \leq K\left(F(\bar{M}_{\tau-1})-F(\bar{M}_{\tau+1})\right).
\end{equation}
Using the discrete Lojasiewicz lemma \cite[Lemma 6.9]{CM_uniqueness} this yields
\begin{equation}
\left(F(\Gamma)-F(\bar{M}_{\tau})\right) \leq C|\tau|^{-1/\eta},
\end{equation}
and 
\begin{equation}\label{tail_decay_sum}
\sum_{j=J}^{\infty} \left(F(\bar{M}_{-j-1})-F(\bar{M}_{-j})\right)^{1/2} \leq C(\nu)J^{-\nu},
\end{equation}
where $\nu = (\eta^{-1}-1)/4$. Since the renormalized mean curvature flow is the downwards gradient flow of the $F$-functional this implies
\begin{equation}
\int_{-\infty}^{\tau}\int_{\bar{M}_{\tau'}} \left|H(q)+\frac{q^{\perp}}{2}\right|e^{-\frac{|q|^2}{4}} d\mu_{\tau'}(q)\, d\tau' \leq C|\tau|^{-\nu}\, .
\end{equation}
Hence, applying \cite[Lemma A.48]{CM_uniqueness} we infer that
\begin{equation}
\int_{\Gamma\cap \{|(y_1,y_2)|\leq \rho(\tau)/2\}}| u(y_1,y_2,\theta,\tau)|e^{-\frac{y_1^2+y_2^2}{4}} \, dy\, d\theta \leq C|\tau|^{-\nu}.
\end{equation}
Together with Proposition \ref{improved radius} (admissibility) and the weighted $L^2$-estimate from \cite[Proposition 4.4]{CHH_wing}
we conclude that
\begin{equation}
\beta(\tau)\leq C|\tau|^{-\nu/2}.
\end{equation}
Choosing $\gamma=\nu/20$, this proves the assertion.
\end{proof}

\bigskip

\subsection{Almost symmetry}
The goal of this subsection is to prove that the angular derivative $ u_\theta$ is exponentially small. To show this we will generalize the arguments from \cite[Section 4.3.3]{CHH} to the bubble-sheet setting.

As in the previous subsection we assume that the neutral mode is dominant. Thanks to Proposition \ref{radius_lower_bound} (graphical radius) we can from now on work with the graphical radius
\begin{equation}
\rho(\tau):=|\tau|^\gamma.
\end{equation}

We will use the  bubble-sheet improvement theorem by Zhu \cite{Zhu}. We say that a vector field $K$ in $\mathbb{R}^{4}$ is a \emph{normalized rotational vector field} if it can be expressed as $K(x)=SJS^{-1}(x-q)$, where $S\in \mathrm{SO}(4)$, $q\in \mathbb{R}^{4}$ and  
\begin{equation}
J=\left(
  \begin{array}{ccccc}
   0 & 0 & 0 & 0  \\
   0 & 0 & 0 & 0   \\
   0 & 0 & 0 & -1    \\
  0 & 0 & 1 & 0  \\
  \end{array}
\right)
\end{equation}
We say that a space-time point $X=(x,t)\in \mathcal{M}$ is  \emph{$\varepsilon$-symmetric} if there exists a normalized rotational vector field $K$ so that
\begin{equation}
|K|H\leq 5 \,\,\, \textrm{and}\,\, \, |\langle K, \nu\rangle|H\leq \varepsilon   \qquad \textrm{in}\quad P(X, 100/ H(X)).
\end{equation}
By \cite[Theorem 3.7]{Zhu} there exist constants $\eps_0>0$ and $L<\infty$ with the following significance. Given $\varepsilon\leq\varepsilon_{0}$ and $X\in \mathcal{M}$, if every $X'\in \mathcal{M}\cap P(X, L/H(X))$ is $\varepsilon_{0}$-close to a bubble-sheet and $\varepsilon$-symmetric,  then $X$ is $\eps/2$-symmetric.

\begin{lemma}[{circular improvement, cf.  \cite[Lemma 4.23]{CHH}}]\label{improvement}
There exists a constant $\eta_{0}>0$, so that for $\tau\ll0$ all space-time points $X\in \mathcal{M}$ corresponding to points in $\bar{M}_{\tau}\cap \{|(y_1,y_2)|\leq \rho(\tau)\}$ are $2^{-\eta_{0}\rho(\tau)}$-symmetric.
\end{lemma}

\begin{proof}
Let $\varepsilon_{0}>0$, $L<\infty$ be the constants from Zhu's bubble-sheet improvement theorem as recalled above. Let $\tau_0\leq\tau_\ast$ be sufficiently negative, and set $R:=2\rho(\tau_0)$.\\
Let $q\in \bar{M}_{\tau}$ be a point with $|(y_1,y_2)|\leq R-L$ and $\tau \leq \tau_0$, and denote by $X\in \mathcal{M}$ the corresponding space-time point in the unrescaled flow. Using \eqref{small_graph} we see that every $X'\in \mathcal{M}\cap P(X,LH^{-1}(X))$ is $\eps_0$-close to a bubble-sheet and is $\eps_0$-symmetric. Hence, by Zhu's bubble-sheet improvement theorem \cite[Theorem 3.7]{Zhu} the space-time point $X$ is $\eps_0/2$-symmetric.\\
Similarly, for any $q\in \bar{M}_{\tau}$ with $|(y_1,y_2)| \leq R-2L$ and $\tau\leq \tau_0$, the argument above shows that the corresponding space-time point $X\in \mathcal{M}$ is $\eps_0/4$-symmetric. Iterating this $k$ times, where $k$ is the smallest integer so that $(k+1)L>R$, we  get that for all $q\in \bar{M}_{\tau}$ with $|(y_1,y_2)| \leq R/2$ and $\tau\leq \tau_0$ the corresponding space-time point $X\in\mathcal{M}$ is $\eps_0/2^k$-symmetric.  Choosing $\eta_0=1/L$, this proves the lemma.
\end{proof}

We can now prove the main result of this subsection:

\begin{proposition}[{almost circular symmetry, cf.  \cite[Proposition 4.24]{CHH}}]\label{utheta}
There exists a constant $\eta>0$, such  that for all $\tau\ll 0$ we have
\begin{equation}
    |u_{\theta}(y_1,y_2,\theta,\tau)|\leq e^{-\eta \rho(\tau)},
\end{equation}
whenever $|(y_1,y_2)|\leq \rho(\tau)$.
\end{proposition}

\begin{proof}
For each $\tau\leq \tau_\ast$ choose a point $q_\tau\in \bar{M_\tau}$ with $y_1=y_2=0$. By Lemma \ref{improvement} (circular improvement) the corresponding space-time point $X_\tau=(x_\tau,-e^{-\tau})\in\mathcal{M}$ in the unrescaled flow is $2^{-\eta\rho(\tau)}$-symmetric, i.e. there exists a normalized rotation vector field $K_\tau$  so that
\begin{equation}
|K_\tau|H\leq 5 \,\,\, \textrm{and}\,\, \, |\langle K_\tau, \nu\rangle|H\leq 2^{-\eta_0\rho(\tau)}   \qquad \textrm{in}\quad P(X_\tau, 100/ H(X_\tau)).
\end{equation}
By \cite[Lemma 3.5]{Zhu} and since $S\in \mathrm{SO}(4)$ by our convention, we have
\begin{equation}
   \sup_{x\in B_{10}(x_{\tau})} |K_\tau(x) - K_{\tau-1}(x)|\leq C2^{-\eta_0\rho(\tau)}.
\end{equation}
Moreover, since the tangent-flow at $-\infty$ is a bubble-sheet, the vector fields $K_\tau(x)$ converge for $\tau\to -\infty$ to $K_{\infty}(x)=Jx$. Hence, we infer that
\begin{equation}
\sup_{x\in B_{10}(x_{\tau})} |K_\tau(x) - K_{\infty}(x)| \leq  C\sum_{m=0}^{\infty}2^{-\eta_0 \rho(\tau-m)}.
\end{equation}
Recalling that $\rho(\tau)=|\tau|^\gamma$, this sum can be safely bounded by $Ce^{-\eta \rho(\tau)}$, where $\eta:=\eta_0/10$.\\
Similarly, as for any $q\in \bar{M}_{\tau}$ with $|(y_1,y_2)|\leq \rho(\tau)$ and $\tau\leq \tau_\ast$ the corresponding space-time point $X_{(q,\tau)}$ is
$C/2^{\eta_0 \rho(\tau)}$-symmetric with normalized rotation vectorfield $K_{(q,\tau)}$, we infer that 
\begin{equation}
\sup_{x\in B_{10}(x_{\tau})} |K_{(q,\tau)}(x) - K_{\infty}(x)|\leq C\rho(\tau) 2^{-\eta_0\rho(\tau) }\leq  Ce^{-\eta \rho(\tau)}.
\end{equation}
Hence, $X_{(q,\tau)}$ is $C/2^{\eta \rho(\tau)}$-symmetric with respect to $K_\infty$. We conclude that
\begin{equation}
\tfrac{1}{2}|u_{\theta}| \leq H |\langle K_\infty, \nu \rangle |\leq Ce^{-\eta \rho(\tau)}.
\end{equation}
Slightly decreasing $\eta$, this proves the proposition.
\end{proof}

\bigskip

\subsection{Evolution expansion}
As before, we assume that the neutral mode is dominant and work with the graphical radius $\rho(\tau)=|\tau|^\gamma$. In particular, we define $\hat{u}$ using this $\rho$.
The goal of this subsection is to prove:

\begin{proposition}[evolution expansion]\label{evolution equation u2} The function $\hat{u}$ evolves by
\begin{align}
   \partial_{\tau}\hat{u}=\mathcal{L}\hat{u}-\frac{1}{\sqrt{8}}\hat{u}^2-\tfrac{1}{\sqrt{8}}   \hat{u}_\theta^2-\tfrac{1}{ \sqrt{2}} \hat{u}  \hat{u}_{\theta\theta}+\hat{E},
\end{align}
where the error term satisfies the weighted $\mathcal{H}$-norm estimate
\begin{equation}
\left\langle |\hat{E}|, 1+y_1^2+y_2^2\right\rangle_{\mathcal{H}}\leq C\rho^{-1}\|\hat{u} \|_{\mathcal{H}}^2+ e^{-\rho/5}\, .
\end{equation}
\end{proposition}

\begin{proof}
By Proposition \ref{MCF_graphical_equation} (evolution over cylinder) the renormalized graphical mean curvature evolution of $u(\cdot,\tau)$ over the cylinder $\Gamma=\mathbb{R}^2\times S^{1}(\sqrt{2})$ is given by
\begin{align}\label{equation_u} 
    \partial_{\tau}u=& \frac{\sum_{ \alpha,\beta=1}^2 A_{\alpha\beta}\partial_{\alpha}\partial_{\beta}u+(1+|\partial u|^2) u_{\theta\theta}-2\sum_{\alpha=1}^2u_\theta \partial_{\alpha} u  \partial_{\alpha}u_\theta}{(1+|\partial u|^2)(\sqrt{2}+u)^2+u_\theta^2}\nonumber\\
    &\frac{-(\sqrt{2}+u)^{-1}u_\theta^{2}}{(1+|\partial u|^2)(\sqrt{2}+u)^2+u_\theta^2}-\frac{1}{\sqrt{2}+u}+\frac{1}{2}\left(\sqrt{2}+u- \sum_{\alpha=1}^2 y_\alpha \partial_\alpha u\right),
\end{align}
where 
\begin{equation}\label{A}
    A_{\alpha\beta}=[(1+|\partial u|^2)(\sqrt{2}+u)^2+u_\theta^2]\delta_{\alpha\beta}-(\sqrt{2}+u)^{2}\partial_{\alpha}u \partial_{\beta}u.
\end{equation}
This implies
\begin{align}
   \partial_\tau u=\mathcal{L}  u+\mathcal{Q}(u)+E,
\end{align}
where
\begin{equation}
\mathcal{Q}(u)=-\tfrac{1}{\sqrt{8}}  u^2-\tfrac{1}{\sqrt{8}}   u_\theta^2-\tfrac{1}{ \sqrt{2}} u   u_{\theta\theta},
\end{equation}
and
\begin{equation}
|E|\leq C|\nabla u|^2(|u|+|\nabla u|+|\nabla^2 u|)+C|u|^2|\nabla^2u|+C|u|^3.
\end{equation}
Next, concerning the cutoff function, using $|\dot{\rho}|\leq \rho$ we compute
\begin{equation}\label{lin_est}
\begin{aligned}
\big|(\partial_\tau-\mathcal{L})( \chi u)-\chi (\partial_\tau -\mathcal{L})u \big| \leq 
\rho^{-2}|\chi''||u|+2\rho^{-1}|\chi'|(|\nabla u|+|ru|),
\end{aligned}
\end{equation}
where $r=|(y_1,y_2)|$. Similarly, using also that $\rho$ is $\theta$-independent we get
\begin{equation}\label{quad_est}
|\mathcal{Q}(\chi u)-\chi\mathcal{Q}(u)|=|\chi^2\mathcal{Q}(u)-\chi\mathcal{Q}(u)|= \chi (1-\chi)|\mathcal{Q}(u)|.
\end{equation}
Putting everything together, this shows that
\begin{equation}
\partial_\tau \hat u = \mathcal{L}\hat u-\tfrac{1}{2\sqrt{2}}\hat u^2-\tfrac{1}{2\sqrt{2}} \hat u_\theta^2-\tfrac{1}{ \sqrt{2}}\hat u \hat u_{\theta\theta} +\hat{E},
\end{equation}
where the error term satisfies the pointwise estimate
\begin{align}\label{est_eq}
|\hat{E}|\leq &C\chi(|u|+|\nabla u|)^2(|u|+|\nabla u|+|\nabla^2u|)\nonumber\\
&+ C |\chi'|\rho^{-1} \big(|\nabla u|+|ru|\big)+ C |\chi''|\rho^{-2}|u|\\
&+ C \chi(1-\chi)\big(|u|^2+|\nabla u|^2+|\nabla^2u|^2\big).\nonumber
\end{align}
To bound this in the $(1+r^2)$-weighted $\mathcal{H}$-norm, first observe that $\hat{E}$ is supported in the ball $\{r\leq \rho\}$, so in particular by the definition of graphical radius we have the inequality $|u|+|\nabla u| + |\nabla^2 u| \leq \rho^{-2}$
at our disposal.
Using this, for $r\leq \rho^{1/2}$ we can estimate 
\begin{equation}
|\hat{E}|(1+r^2)\leq  C\rho^{-1} \left(|u|^2 + |\nabla u|^2\right).
\end{equation}
Together with \cite[Proposition 4.4]{CHH_wing} this implies
\begin{equation}
\int_{r\leq \rho^{1/2}}|\hat{E}|(1+r^2)\, e^{-\frac{r^2}{4}} \, dy_1dy_2 \leq C\rho^{-1}\|\hat{u}\|_{\mathcal{H}}^2\, .
\end{equation}
Finally, for $r\geq \rho^{1/2}$ the coarse bound $|\hat{E}|(1+r^2) \leq C \rho^2$ yields the tail estimate
\begin{align}
\int_{\rho/2\leq r\leq \rho}|\hat{E}|(1+r^2) \, e^{-\frac{r^2}{4}}\, dy_1dy_2  \leq  e^{-\rho/5} \, .
\end{align}
This finishes the proof of the proposition.
\end{proof}

\bigskip

\section{Proof of the bubble-sheet quantization theorem}\label{sec_quant_thm}
In this section, we prove Theorem \ref{spectral theorem} (bubble-sheet quantization). As before, we work with the graphical radius $\rho(\tau)=|\tau|^\gamma$, and in particular define $\hat{u}$ and $U_\pm,U_0$ with respect to this $\rho$. By 
Proposition \ref{mz.ode.fine.bubble-sheet} (Merle-Zaag alternative) for $\tau\to -\infty$ either the neutral mode is dominant, i.e.
$U_-+U_+=o(U_0)$, 
or the unstable mode is dominant, i.e.
$U_-+U_0\leq C\rho^{-1}U_+$.\\

If the unstable mode is dominant, then by \eqref{U_PNM_system} we get
$\dot{U}_{+} \geq \tfrac{1}{2}U_{+}$  for all $\tau\ll 0$. 
Integrating this differential inequality yields $U_+(\tau) \leq C e^{\tau/2}$. Thus, recalling that $U_+ = \| \mathcal{P}_+ \hat{u}\|_{\mathcal{H}}^2$ and using again the assumption that the unstable mode is dominant we infer that $\| \hat{u} \|_{\mathcal{H}} \leq Ce^{\tau/4}$. Together with standard interpolation inequalities this implies that for any $R<\infty$ and $k<\infty$  we have
$\| u(\cdot,\tau) \|_{C^{k}(B_R)} \leq Ce^{\tau/5}$. Hence, if the the unstable mode is dominant then the conclusion of Theorem \ref{spectral theorem} holds with $Q=0$. We can thus assume from now on that the neutral mode is dominant, i.e.
\begin{equation}\label{neu_dom}
U_-+U_+=o(U_0).
\end{equation}
Assuming also that the flow is not a round shrinking cylinder our goal is to show that the conclusion of Theorem \ref{spectral theorem} holds for some $Q$ with $\mathrm{rk}(Q)\geq 1$.

\subsection{Derivation of the spectral ODEs}
In this subsection, we consider the expansion
\begin{equation}\label{expansion_123}
 \hat{u} = \sum_{j=1}^3 \alpha_j \psi_j +w,
\end{equation}
with respect to the eigenfunctions 
\begin{equation}\label{phi123}
    \psi_{1}=y^2_{1}-2,\quad  \psi_{2}=y^2_{2}-2,\quad  \psi_{3}=2y_{1}y_{2}.
\end{equation}
Hence, the spectral coefficients $\vec{\alpha}=(\alpha_1,\alpha_2,\alpha_3)$ are given by
\begin{equation}\label{def_exp_coeffs}
\alpha_j = \|\psi_j\|_{\mathcal{H}}^{-2} \langle  \psi_j,\hat{u}, \rangle_{\mathcal{H}}.
\end{equation}
The goal is to derive the following ODEs for the spectral coefficients:
\begin{proposition}[spectral ODEs]\label{odes-1}
The spectral coefficients $\alpha_{j}(\tau)$ satisfy the following system of ODEs:
\begin{equation}\label{odes0}
 \begin{cases}
   \dot{\alpha}_{1}=-\sqrt{8}(\alpha^2_{1}+\alpha_{3}^2)+o(|\vec{\alpha}|^2+|\tau|^{-100})\\
   \dot{\alpha}_{2}=-\sqrt{8}(\alpha^2_{2}+\alpha_{3}^2)+o(|\vec{\alpha}|^2+|\tau|^{-100})\\
    \dot{\alpha}_{3}=-\sqrt{8}(\alpha_{1}+\alpha_2)\alpha_{3}+o(|\vec{\alpha}|^2+|\tau|^{-100})\\
    \end{cases}
\end{equation}
\end{proposition}

To show this, we start with the following remainder estimate:

\begin{lemma}[remainder estimate]\label{rotationdecay}
The remainder $w$ satisfies the weighted $\mathcal{H}$-norm estimate
\begin{equation}\label{remainder1}
\left\langle w^2, 1+y_1^2+y_2^2 \right\rangle_{\mathcal{H}} =  o(|\vec{\alpha}|^2+e^{-\eta\rho/3}).
\end{equation}
\end{lemma}
\begin{proof} Recall that $\mathcal{H}_0$ is 7-dimensional, spanned by the 3 eigenfunctions from \eqref{phi123} together with the following 4 eigenfunctions:
\begin{equation}\label{phi4567}
    \psi_{4}=y_{1}\cos\theta\quad \psi_{5}=y_{1}\sin\theta\quad \psi_{6}=y_{2}\cos\theta\quad \psi_{7}=y_{2}\sin\theta \, .
\end{equation}
If we consider the remainder of the full $\mathcal{H}_0$-expansion,
\begin{equation}
\tilde{w}:= \hat{u}- \sum_{j=1}^7 \alpha_j \psi_j, \qquad \textrm{with} \quad \alpha_j = \|\psi_j\|_{\mathcal{H}}^{-2}\langle \psi_j ,\hat{u} \rangle_{\mathcal{H}},
\end{equation}
where now $j=1,\ldots, 7$, then by assumption \eqref{neu_dom} we have
\begin{equation}
\|\tilde{w}\|_{\mathcal{H}}^2 = o\left( \sum_{j=1}^7 \alpha_j^2\right).
\end{equation}
On the other hand, using the relation $\psi_{4}= \partial_\theta \psi_{5}$ we can estimate
\begin{equation}
|\langle\psi_4,\hat{u}\rangle_{\mathcal{H}}|= |\langle  \psi_5,\partial_\theta \hat{u}\rangle_{\mathcal{H}} | \leq Ce^{-\eta\rho},
\end{equation}
where we applied Proposition \ref{utheta} (almost circular symmetry) in the last step.
Arguing similarly for $\alpha_5,\ldots,\alpha_7$, it follows that
\begin{equation}\label{remainder1}
\|w\|_{\mathcal{H}}^2 = o(|\vec{\alpha}|^2) + Ce^{-\eta\rho}.
\end{equation}
We will next prove a gradient estimate for $w$. To this end, we project the equation from Proposition \ref{evolution equation u2} (evolution equation) to the orthogonal complement of $\mathrm{span}\{ \psi_1,\psi_2,\psi_3\}$, and argue as above to obtain
\begin{equation}\label{eq_w}
\partial_\tau w = \mathcal{L} w + g,
\end{equation}
where
\begin{equation}\label{est_g}
\| g\|_{\mathcal{H}}^2= o(|\vec{\alpha}|^2) + Ce^{-\eta\rho}.
\end{equation}
Now, given $\tau_\ast\ll 0$, using \eqref{eq_w} and integration by parts we compute
\begin{align}
\frac{d}{d\tau}\int {e^{\tau_\ast-\tau}} w^2e^{-q^2/4} &= \int  {e^{\tau_\ast-\tau}}(2wg-2|\nabla w|^2 -w^2)\, e^{-q^2/4}\nonumber\\
& \leq \int  {e^{\tau_\ast-\tau}}(g^2-2|\nabla w|^2)\, e^{-q^2/4},
\end{align}
and
\begin{align}
 \frac{d}{d\tau}\int (\tau-\tau_\ast)|\nabla w|^2 e^{-q^2/4}
 &=\int \left(|\nabla w|^2  -2(\tau-\tau_\ast) (\mathcal{L}  w)(\mathcal{L}w+g)\right)\, e^{-q^2/4}\nonumber\\
 &\leq\int \left(|\nabla w|^2  +\tfrac{1}{2}(\tau-\tau_\ast) g^2)\right)\, e^{-q^2/4}\, .
\end{align}
For $\tau\in [\tau_\ast-1,\tau_\ast]$ this yields
\begin{align}
 \frac{d}{d\tau}\int \left((\tau-\tau_\ast)|\nabla w|^2+\tfrac{e^{\tau_\ast-\tau}}{2} w^2\right)\, e^{-q^2/4}\leq \int g^2\, e^{-q^2/4}\, .
 \end{align}
Hence, together with \eqref{remainder1} and \eqref{est_g} we infer that
\begin{equation}
\|\nabla w(\tau)\|_{\mathcal{H}}^2 = o\left(\max_{\tau' \in [\tau,\tau+1]}|\vec{\alpha}(\tau')|^2\right) + Ce^{-\eta\rho(\tau)/2}.
\end{equation}
Finally, by the Merle-Zaag ODEs \eqref{U_PNM_system} for $\tau\ll 0$ we have
\begin{equation}
\max_{\tau' \in [\tau,\tau+1]}|\vec{\alpha}(\tau')|^2\leq 2 |\vec{\alpha}(\tau)|^2\, .
\end{equation}
Together with Ecker's weighted Sobolev inequality \cite[page 109]{Ecker_logsob}, this implies the assertion.
\end{proof}

We can now prove the main result of this subsection:

\begin{proof}[{Proof of Proposition \ref{odes-1}}]
Using Proposition \ref{evolution equation u2} (evolution expansion) we get
\begin{align}\label{ak}
\dot{\alpha}_k=  \|\psi_k\|_{\mathcal{H}}^{-2}\left\langle  \mathcal{L}\hat{u}-\frac{1}{\sqrt{8}}\hat{u}^2-\tfrac{1}{\sqrt{8}}   \hat{u}_\theta^2-\tfrac{1}{ \sqrt{2}} \hat{u}  \hat{u}_{\theta\theta}+\hat{E},\psi_k\right\rangle_{\mathcal{H}}\, ,
\end{align}
where
\begin{equation}
|\langle\hat{E}, \psi_k \rangle_{\mathcal{H}}| \leq C\left(\rho^{-1}\|\hat{u} \|_{\mathcal{H}}^2+ e^{- \rho/5}\right)\, .
\end{equation}
Next, using $\mathcal{L}\psi_k =0$ and integration by parts we see that
\begin{equation}
\langle \mathcal{L}\hat{u},\psi_k \rangle_{\mathcal{H}} = 0.
\end{equation}
Moreover, using Proposition \ref{utheta} (almost circular symmetry), writing $\hat{u}\hat{u}_{\theta\theta}=(\hat{u}\hat{u}_\theta)_\theta-\hat{u}_\theta^2$, and using integration by parts we can estimate
\begin{equation}
|\langle \hat{u}_\theta^2,\psi_k \rangle_{\mathcal{H}}| + |\langle \hat{u}\hat{u}_{\theta\theta},\psi_k \rangle_{\mathcal{H}}|  \leq  C e^{-\eta\rho}.
\end{equation}
Furthermore, remembering the expansion \eqref{expansion_123} we compute
\begin{equation}
\langle \hat{u}^2,\psi_k \rangle_{\mathcal{H}} = \sum_{i,j=1}^3 \alpha_i\alpha_j \langle \psi_i \psi_j, \psi_k \rangle_\mathcal{H} + 2 \sum_{i=1}^3 \alpha_i \langle \psi_i \psi_k ,w\rangle_\mathcal{H}+\langle w^2, \psi_k \rangle_\mathcal{H}\, .
\end{equation}
By Lemma \ref{rotationdecay} (remainder estimates) we have
\begin{equation}
 2 \sum_{i=1}^3 \alpha_i \langle \psi_i \psi_k ,w\rangle_\mathcal{H}+\langle w^2, \psi_k \rangle_\mathcal{H} = o(|\vec{\alpha}|^2+e^{-\eta\rho/3})\, .
\end{equation}
Combining the above facts we infer that
\begin{equation}
\dot{\alpha}_k= -\frac{1}{\sqrt{8}}\sum_{i,j=1}^3  \|\psi_k\|_{\mathcal{H}}^{-2} \langle \psi_i \psi_j, \psi_k\rangle_{\mathcal{H}} \alpha_i \alpha_j + o(|\vec{\alpha}|^2+|\tau|^{-100}).
\end{equation}
Finally, for functions on our bubble-sheet $\Gamma=\mathbb{R}^2\times S^1(\sqrt{2})$ that are independent of $\theta$ the Gaussian inner product is explicitly given by
\begin{equation}
    \langle f, g\rangle_{\mathcal{H}}= (8e\pi)^{-1/2}\int_{\mathbb{R}^2} e^{-\frac{y_1^2+y_2^2}{4}}f(y_1,y_2)g(y_1,y_2) \, dy_1dy_2\, .
\end{equation}
Hence, an elementary computation shows that
\begin{equation}\label{I1}
     \|\psi_{1}\|^2_{\mathcal{H}}=\|\psi_{2}\|^2_{\mathcal{H}}=\tfrac{1}{2} \|\psi_{3}\|^2_{\mathcal{H}} \,
\end{equation}
and that up to permutation of indices the only nonvanishing coefficients are
\begin{equation}\label{I2}
    \langle \psi_1\psi_1, \psi_1 \rangle_\mathcal{H}=\langle \psi_2\psi_2, \psi_2 \rangle_\mathcal{H}=8 \|\psi_{1}\|^2_{\mathcal{H}} ,
\end{equation}
and
\begin{equation}\label{I3}
     \langle\psi_3\psi_3,\psi_{1}\rangle_{\mathcal{H}}= \langle \psi_{3}\psi_3,\psi_{2}\rangle_{\mathcal{H}}=4 \|\psi_{3}\|^2_{\mathcal{H}}.
\end{equation}
Putting things together, this proves the proposition.
\end{proof}

\bigskip

\subsection{Quantized asymptotics of the spectral ODEs} Recall from Proposition  \ref{odes-1} (spectral ODEs) that the spectral coefficients satisfy
\begin{equation}\label{odes0_rest}
 \begin{cases}
   \dot{\alpha}_{1}=-\sqrt{8}(\alpha^2_{1}+\alpha_{3}^2)+o(|\vec{\alpha}|^2+|\tau|^{-100})\\
   \dot{\alpha}_{2}=-\sqrt{8}(\alpha^2_{2}+\alpha_{3}^2)+o(|\vec{\alpha}|^2+|\tau|^{-100})\\
    \dot{\alpha}_{3}=-\sqrt{8}(\alpha_{1}+\alpha_2)\alpha_{3}+o(|\vec{\alpha}|^2+|\tau|^{-100})\\
    \end{cases}
\end{equation}
To solve these spectral ODEs, we start with the following a priori estimate:

\begin{proposition}[a priori estimate]\label{conv_est}
The spectral coefficients satisfy
\begin{equation}
\alpha_1 \leq o(|\vec{\alpha}|+|\tau|^{-100}), \quad \alpha_2 \leq o(|\vec{\alpha}|+|\tau|^{-100}),
\end{equation}
and
\begin{equation}
\alpha_3^2 - \alpha_1\alpha_2 \leq o(|\vec{\alpha}|^2+|\tau|^{-100}).
\end{equation}
\end{proposition}

\begin{proof}
In essence, this will be a consequence of \cite[Proposition 5.2]{CHH_wing}. However, since the setup of the present paper is different we first have to related the tilted and untilted flow. To discuss this, recall that in the cited proposition the tilted flow $\tilde{M}_\tau = S(\tau)\bar{M}_\tau$ has been considered, where the fine-tuning rotation $S(\tau)\in \mathrm{SO}(4)$ has been constructed in \cite[Proposition 4.1]{CHH_wing} via the implicit function theorem to ensure that the profile function $\tilde{u}$ of the tilted flow satisfies the orthogonality conditions
\begin{equation}
\langle \tilde{u}\chi,\psi_j \rangle_{\mathcal{H}} =0 \qquad (j=4,\ldots,7).
\end{equation}
Observe that thanks to Proposition \ref{utheta} (almost circular symmetry) the profile function $u$ of the untilted flow already satisfies
\begin{equation}
|\langle u\chi,\psi_j \rangle_{\mathcal{H}} | \leq Ce^{-\eta \rho} \qquad (j=4,\ldots,7).
\end{equation}
Hence, inspecting the proof of \cite[Proposition 4.1]{CHH_wing} we can arrange that
\begin{equation}\label{tilt_small}
|S(\tau) - \textrm{id}| \leq Ce^{-\eta \rho}.
\end{equation}
Now, suppose towards a contradiction there is a sequence $\tau_i\to -\infty$ with 
\begin{equation}\label{seq_large}
\liminf_{i\to \infty}   |\tau_i|^{100}|\vec{\alpha}(\tau_i)| > 0,
\end{equation}
such that
\begin{equation}\label{eq_at_least_one}
\max\left\{\liminf_{i \to \infty} \frac{ \alpha_1}{|\vec{\alpha}|}({\tau_i}), \liminf_{i \to \infty} \frac{ \alpha_2}{|\vec{\alpha}|}({\tau_i}), \liminf_{i \to \infty} \frac{ \alpha_3^2- \alpha_1 \alpha_2}{|\vec{\alpha}|^2}(\tau_i)\right\}>0.
\end{equation}
By \cite[Proposition 5.2]{CHH_wing} after passing to a subsequence we have
\begin{equation}
\lim_{i'\to \infty} \frac{\hat{\tilde{u}}(\cdot,\tau_{i'})}{\|\hat{\tilde{u}}(\cdot,\tau_{i'})\|_{\mathcal{H}} } = q_{11} \psi_1(y)+q_{22}\psi_2(y)+q_{12}\psi_3(y)
\end{equation}
in $\mathcal{H}$-norm, where $\{q_{\alpha\beta}\}$ is semi-negative definite symmetric $2\times2$-matrix.
Thanks to \eqref{tilt_small} and \eqref{seq_large} this holds for the untilted flow as well, i.e.
\begin{equation}
\lim_{i'\to \infty} \frac{\hat{u}(\cdot,\tau_{i'})}{\|\hat{u}(\cdot,\tau_{i'})\|_{\mathcal{H}} } = q_{11} \psi_1(y)+q_{22}\psi_2(y)+q_{12}\psi_3(y).
\end{equation}
This contradicts \eqref{eq_at_least_one}, and thus proves the proposition.
\end{proof}

We now consider the trace and determinant,
\begin{equation}
    S:=\alpha_{1}+\alpha_{2},\quad D:=\alpha_{1} \alpha_{2}-{\alpha_{3}^{2}}  \, .
\end{equation}

\begin{proposition}[evolution of trace and determinant]\label{prop_sum_disc}
The trace and determinant satisfy
\begin{equation}\label{Riccati ODEs}
    \begin{cases}
      \dot{S}=-\sqrt{8}(S^2-2D)+o(|\tau|^{-2}),\\
       \dot{D}=-\sqrt{8}SD+o(|\tau|^{-3}).
    \end{cases}
    \end{equation}
\end{proposition}

\begin{proof}
Using Proposition \ref{odes-1} (spectral ODEs) we compute
\begin{align}
      \dot{S}&=-\sqrt{8}(\alpha_1^2+\alpha_2^2 + 2\alpha_3^2)+o(|\vec{\alpha}|^2+|\tau|^{-100})\nonumber\\
      &= -\sqrt{8}(S^2-2D)+o(|\vec{\alpha}|^2+|\tau|^{-100}),
\end{align}      
      and
      \begin{align}
       \dot{D}&=-\sqrt{8}\left( \alpha_1 \alpha_2^2+\alpha_1\alpha_3^2+\alpha_2  \alpha_1^2+\alpha_2\alpha_3^2 -2S\alpha_3^2  \right)+o(|\vec{\alpha}|^3+|\tau|^{-100})\nonumber\\
       &=-\sqrt{8}SD+o(|\vec{\alpha}|^3+|\tau|^{-100}).
\end{align}  
Next, observe that Proposition \ref{conv_est} (a priori estimate) implies
\begin{equation}\label{s_dom_est}
|\vec{\alpha}|^2 \leq (1+o(1)) S^2+o(|\tau|^{-100}).
\end{equation}
This yields
 \begin{equation}\label{sum_disc_int}
    \begin{cases}
      \dot{S}=-\sqrt{8}(S^2-2D)+o(S^2+|\tau|^{-100}),\\
       \dot{D}=-\sqrt{8}SD+o(|S|^3+|\tau|^{-100}).
    \end{cases}
 \end{equation}  
To proceed, we need the following claim:

\begin{claim}[trace asymptotics]\label{claim_asympt_sum} We have
\begin{equation}\label{Stau}
   \frac{1+o(1)}{\sqrt{2}\tau} \leq S\leq  \frac{1-o(1)}{2\sqrt{2}\tau}.
\end{equation}
\end{claim}
\begin{proof}[Proof of the claim]
Note that Proposition \ref{conv_est} (a priori estimate) implies
\begin{equation}\label{Dtau}
-o(S^2+|\tau|^{-100})\leq D\leq \frac{1}{4}S^2.
\end{equation}
This yields
\begin{equation}\label{the_ode_for_s}
    -2\sqrt{2}S^2-o(S^2+|\tau|^{-100}) \leq \dot{S}\leq -\sqrt{2}S^2+o(S^2+|\tau|^{-100}).
\end{equation}
To control the errors, we consider the monotone quantity
\begin{equation}\label{def_bar_s}
    \bar{S}(\tau)=\sup_{\sigma\leq\tau}|S(\sigma)|\, .
\end{equation}
We first observe that
\begin{equation}\label{tau-S}
    \limsup_{\tau\rightarrow -\infty}|\tau|^{10}\bar{S}(\tau)=\infty\, .
\end{equation}
Indeed, if this failed then using in particular \eqref{s_dom_est} we could infer that the function $\rho(\tau)=c|\tau|^2$, where $c>0$ is small, is an admissible graphical radius, c.f.  Proposition \ref{prop_admiss} (admissibility). However, by \eqref{U_PNM_system} this would imply $|\tfrac{d}{d\tau} \log U_0(\tau)| \leq C|\tau|^{-2}$, hence $|\log U_0(\tau)| \leq C$ for $\tau\ll 0$. This contradicts the fact that $U_0(\tau)$ converges to $0$ for $\tau\to -\infty$, and thus proves \eqref{tau-S}.\\
We now fix $\tau_\ast\ll 0$, and consider the sets
\begin{align}
I:=\left\{\tau \leq \tau_\ast \, :\,  \bar{S}(\tau)=|S(\tau)|\right\}, \quad J:&=\left\{\tau \leq \tau_\ast \, :\, \bar{S}(\tau)\geq |\tau|^{-10}\right\}.
\end{align}
Note that thanks to $\lim_{\tau\to-\infty}\bar{S}(\tau)= 0$ and \eqref{tau-S} the set $I\cap J$ contains a sequence of numbers going to $-\infty$. Also, clearly $I\cap J\subseteq (-\infty,\tau_\ast]$ is closed. Now, given any $\tau_0\in I\cap J$ by \eqref{the_ode_for_s} we have
\begin{equation}\label{OED_ineq}
\dot{S}(\tau)\leq -S^2(\tau)\leq -\tfrac{1}{2}|\tau|^{-100}
\end{equation}
at $\tau=\tau_0$. In particular, remembering \eqref{def_bar_s}, we see that $S(\tau_0)<0$. In fact, we can find a $\delta>0$ such that \eqref{OED_ineq} and $S(\tau)<0$ hold for $|\tau-\tau_0|<\delta$.
Moreover, if there is some $\hat{\tau} \in I$ with $\hat\tau \leq\tau_0$ and $\dot{S}(\hat{\tau})=0$ then \eqref{the_ode_for_s} yields
\begin{equation}
|S^2(\hat{\tau})|\leq 2|\hat{\tau}|^{-100}\leq 2|{\tau}_0|^{-100}< \tfrac{1}{100} |S^2(\tau_0)|.
\end{equation}
Thus, if  $ \frac{1}{10}|S(\tau_0)|\leq |S(\tau)|\leq |S(\tau_0)|$ then $\tau\in I$.
Hence, possibly after decreasing $\delta$, we get $(\tau_0-\delta,\tau_0]\subseteq I$. Then, \eqref{OED_ineq}  implies $(\tau_0-\delta,\tau_0]\subseteq J$.\\
Summarizing, we can find $S$ by solving 
\begin{equation}
    -2\sqrt{2}S^2-o(S^2) \leq \dot{S}\leq -\sqrt{2}S^2+o(S^2)
\end{equation}
for $\tau\leq\tau_\ast$, subject to $\lim_{\tau\to -\infty}S(\tau)=0$. This yields the claim.
\end{proof}

To conclude, note that by the claim we have
\begin{equation}
o(S^2+|\tau|^{-100})=o(|\tau|^{-2}), \qquad o(|S|^3+|\tau|^{-100})=o(|\tau|^{-3}).
\end{equation}
Together with \eqref{sum_disc_int} this finishes the proof of the proposition.
\end{proof}

We can now conclude the proof of the bubble-sheet quantization theorem, which we restate here in a technically sharper way:

\begin{theorem}[bubble-sheet quantization]\label{spectral theorem_restated}
For any ancient noncollapsed mean curvature flow in $\mathbb{R}^{4}$ whose tangent flow at $-\infty$ is given by \eqref{bubble-sheet_tangent_intro}, the bubble-sheet function $u$, truncated at the graphical radius $\rho(\tau)=|\tau|^\gamma$, where $\gamma$ is the exponent from Proposition \ref{radius_lower_bound} (graphical radius), satisfies
 \begin{equation}
\lim_{\tau\to -\infty} \Big\|\, |\tau| \hat{u}(y,\theta,\tau)- y^\top Qy +2\mathrm{tr}(Q)\, \Big\|_{\mathcal{H}} = 0,
\end{equation}
where $Q$ is a symmetric $2\times 2$-matrix whose eigenvalues are quantized to be either 0 or $-1/\sqrt{8}$.
In particular, for all $R<\infty$ and all $k\in\mathbb{N}$ we have
 \begin{equation}
\lim_{\tau\to -\infty} \Big\|\,
|\tau| u(y,\theta,\tau)- y^\top Qy +2\mathrm{tr}(Q)\, \Big\|_{C^{k}(B_R)} = 0.
\end{equation}
Here, for $\mathrm{rk}(Q)\neq 1$ the matrix $Q$ is independent of time, while in the case $\mathrm{rk}(Q)=1$ we have
\begin{equation}
Q=R(\tau)^\top \begin{pmatrix}
0 & 0\\
0 & -1/\sqrt{8}
\end{pmatrix}R(\tau)
\end{equation}
for some rotation matrix $R(\tau)\in \mathrm{SO}(2)$ with $|\dot{R}(\tau)|=o(|\tau|^{-1})$.
 \end{theorem}

\begin{proof}By the reduction at the beginning of this section, assuming that the neutral mode is dominant we have to show that the conclusion holds for for some $Q$ with $\mathrm{rk}(Q)\geq 1$. To get rid of some annoying prefactors and minus signs, we set
\begin{equation}
a:=-\sqrt{2}S,\qquad b:= 8D\, .
\end{equation}
Then, our ODEs from Proposition \ref{prop_sum_disc} (evolution of trace and determinant) take the form
\begin{equation}\label{Riccati ODEs_ab}
    \begin{cases}
      \dot{a}=2a^2-b+o(|\tau|^{-2}),\\
       \dot{b}=2ab+o(|\tau|^{-3}),
    \end{cases}
    \end{equation}
and our a priori estimates from \eqref{Stau} and \eqref{Dtau} take the form
\begin{equation}
\frac{1-o(1)}{2|\tau|} \leq a \leq \frac{1+o(1)}{|\tau|}\, , \qquad \frac{-o(1)}{\tau^2} \leq b \leq \frac{1+o(1)}{\tau^2}\, .
\end{equation}   
In fact, it is useful to make yet another substitution to get rid of the $\tau$-dependence. Specifically, we set
\begin{equation}
x:=-\tau a\, , \qquad y:=\tau^2 b\ , \qquad \sigma:=-\log(-\tau)\, .
\end{equation}
In these new variables our ODEs take the form
\begin{equation}\label{Riccati ODEs_xy}
    \begin{cases}
      x'=2x^2-x-y+o(1),\\
      y'=2xy-2y+o(1),
    \end{cases}
    \end{equation}
and our a priori estimates take the form
\begin{equation}\label{apriori_xy}
\frac{1}{2}-o(1) \leq x \leq 1+o(1)\, , \qquad -o(1)\leq y \leq 1+o(1) \, .
\end{equation}
Observe that the vector field
\begin{equation}
V(x,y)=(2x^2-x-y,2xy-2y)
\end{equation}
in the region relevant by \eqref{apriori_xy} has exactly the two zeros
\begin{equation}
(x,y)=(1/2,0),\qquad (x,y)=(1,1)\, ,
\end{equation}
and that there is an integral curve from $(1,1)$ to $(1/2,0)$, but no integral curve in the other direction.
It follows that for $\sigma\to -\infty$ we either have
\begin{equation}\label{first_option}
x = \frac{1}{2} + o(1)\, , \quad y =o(1)\, ,
\end{equation}
or
\begin{equation}\label{second_option}
x = 1 + o(1)\, , \quad y =1+o(1)\, .
\end{equation}
Indeed, either \eqref{first_option} holds and we are done, or there exists some $\eps>0$ and arbitrarily negative times $\sigma$ such that $|(x(\sigma),y(\sigma))-(1/2,0)|\geq \eps$. But then integrating the ODEs \eqref{Riccati ODEs_xy} backwards in time and using a Lyapunov function argument, we would either leave the relevant rectangular region and thus obtain a contradiction with the a priori estimates \eqref{apriori_xy}, or converge to $(1,1)$ and thus conclude that \eqref{second_option} holds.
Consequently, the trace and determinant for $\tau\to -\infty$ satisfy either 
\begin{equation}
    \begin{cases}
S=\frac{1}{\sqrt{8}\tau}+o(|\tau|^{-1})\\  
D=o(|\tau|^{-2}),
\end{cases}
\end{equation}
or
    \begin{equation}
    \begin{cases}
S=\frac{1}{\sqrt{2}\tau}+o(|\tau|^{-1})\\  
D=\frac{1}{8\tau^2}+o(|\tau|^{-2}).
\end{cases}
\end{equation}
In the first case, remembering also Proposition \ref{conv_est} (a priori estimates), we infer that the eigenvalues of the matrix
\begin{equation}
\begin{pmatrix}
\alpha_1 & \alpha_3\\
\alpha_3 & \alpha_2
\end{pmatrix}
\end{equation}
are
\begin{equation}
\frac{1}{\sqrt{8}\tau}+ o(|\tau|^{-1}) \quad \textrm{and} \quad o(|\tau|^{-1})\, .
\end{equation}
Hence, the conclusion
 \begin{equation}
\lim_{\tau\to -\infty} \Big\|\, |\tau| \hat{u}(y,\theta,\tau)- y^TQy +2\mathrm{tr}(Q)\, \Big\|_{\mathcal{H}} = 0
\end{equation}
holds with
\begin{equation}
Q=R(\tau)^\top \begin{pmatrix}
0 & 0\\
0 & -1/\sqrt{8}
\end{pmatrix}R(\tau)\, ,
\end{equation}
where
\begin{equation}
R(\tau)=\begin{pmatrix}
\cos \phi(\tau) & -\sin \phi(\tau)\\
\sin \phi(\tau) & \cos \phi(\tau)
\end{pmatrix}\, .
\end{equation}
Furthermore, considering the original spectral ODEs,
\begin{equation}\label{odes2}
 \begin{cases}
    \dot{\alpha}_{1}=-\sqrt{8}(\alpha^2_1+\alpha_{3}^2)+o(|\tau|^{-2})\\
    \dot{\alpha}_{2}=-\sqrt{8}(\alpha^2_2+\alpha_{3}^2)+o(|\tau|^{-2})\\
    \dot{\alpha}_{3}=-\sqrt{8}(\alpha_1+\alpha_2)\alpha_3+o(|\tau|^{-2}),\\
    \end{cases}
\end{equation}
we infer that
\begin{equation}
|\dot{R}(\tau)|= o(|\tau|^{-1})\, .
\end{equation}
Finally, in the second case the eigenvalues are
\begin{equation}
\frac{1}{\sqrt{8}\tau}+ o(|\tau|^{-1}) \quad \textrm{and} \quad \frac{1}{\sqrt{8}\tau}+ o(|\tau|^{-1})\, .
\end{equation}
Hence, the conclusion holds with
\begin{equation}
Q= \begin{pmatrix}
-1/\sqrt{8} & 0\\
0 & -1/\sqrt{8}
\end{pmatrix}\, .
\end{equation}
This finishes the proof of the theorem.
\end{proof}

\bigskip

\section{The fully-degenerate case}\label{sec_full_deg}

In this section, we prove Theorem \ref{Rk=0} (fully-degenerate case), which we restate here for convenience of the reader.

\begin{theorem}[fully-degenerate case]\label{Rk=0_rest}
Let $M_t$ be an ancient noncollapsed mean curvature flow in $\mathbb{R}^{4}$ whose tangent flow at $-\infty$ is given by \eqref{bubble-sheet_tangent_intro}. If $\mathrm{rk}(Q)=0$, then $M_t$ is either a round shrinking $\mathbb{R}^2\times S^1$ or $\mathbb{R}\times$2d-bowl. 
\end{theorem}

\begin{proof}
We assume throughout the proof that the flow is not a round shrinking $\mathbb{R}^2\times S^1$. As we have seen in the previous section our assumption $\mathrm{rk}(Q)=0$ is then equivalent to the assumption that in Proposition \ref{mz.ode.fine.bubble-sheet} (Merle-Zaag alternative) the unstable mode is dominant. If the flow is noncompact, then by \cite[Theorem 1.10]{CHH_wing} it must be $\mathbb{R}\times$2d-bowl. Hence, our task is to rule out the compact case.\\
So suppose towards a contraction $\mathcal{M}=\{M_t\}$ is a compact ancient noncollapsed mean curvature flow in $\mathbb{R}^{4}$ whose tangent flow at $-\infty$ is given by \eqref{bubble-sheet_tangent_intro} and for which in the Merle-Zaag alternative the unstable mode is dominant. Then by \cite[Theorem 6.7 and Theorem 6.9]{CHH_wing} there is a universal nonvanishing fine bubble-sheet vector $(a_{1}, a_{2})$ associated to our flow, such that for any space-time point $X$ after suitable recentering in the $x_3x_4$-plane the profile function $u^X$ of the renormalized flow $\bar{M}_\tau^X$ centered at $X$ satsifies
\begin{equation}
u^X = e^{\tau/2}(a_1y_1+a_2y_2) + o(e^{\tau/2})
\end{equation}
for all $\tau\leq \tau_\ast(Z(X))$, depending only on an upper bound for the bubble-sheet scale $Z(X)$ as defined in \cite[Definition 2.8]{CHH_wing}.\\
Now, considering any sequence $t_i\to -\infty$, let $p_{t_i}^\pm\in M_{t_i}$ be points such that
\begin{equation}\label{choice_min_max}
x_1(p_{t_i}^-)=\min_{p\in M_{t_i}} x_1(p)\, ,\qquad x_1(p_{t_i}^+)=\max_{p\in M_{t_i}} x_1(p)\, .
\end{equation}

\begin{claim}[bubble-sheet scale]\label{cyl_scale_bd}
We have
\begin{equation}
\sup_i Z(p^{\pm}_i,t_i)< \infty.
\end{equation}
\end{claim}

\begin{proof}[Proof of the claim]
We will argue similarly as in the proofs of \cite[Proposition 5.8]{CHH}, \cite[Proposition 6.2]{CHHW}, and \cite[Claim 7.2]{CHH_wing}.\\
Write $X_i^\pm=(p^{\pm}_i,t_i)$, and suppose towards a contradiction that $Z(X_i^\pm)\to \infty$. Let $\mathcal{M}^i$ be the sequence of flows obtained by shifting $X_i^\pm$ to the origin, and parabolically rescaling by $Z(X^{\pm}_i)^{-1}$. By \cite[Theorem 1.14]{HaslhoferKleiner_meanconvex} we can pass to a subsequential limit $\mathcal{M}^\infty$, which is an ancient noncollapsed flow that is weakly convex and smooth until it becomes extinct. Note also that by construction $\mathcal{M}^\infty$ has bubble-sheet tangent flow at $-\infty$.\\
Now, by Theorem \ref{spectral theorem} (bubble-sheet quantization) associated $\mathcal{M}^\infty$ there is a fine-bubble sheet matrix $Q^\infty$. If $\textrm{rk}(Q^\infty)\geq 1$, then
 for large $i$ this contradicts the fact that $\mathcal{M}^i$ has dominant unstable mode. Thus, $Q^\infty=0$.\\
Moreover, if $\mathcal{M}^\infty$ was a round shrinking $\mathbb{R}^2\times S^1$, then if it became extinct at time $0$ that would contradict the definition of the bubble-sheet scale, and if it became extinct at some later time that would contradict the fact that $M^\infty_0\cap (\mathbb{R}^2\times \{0\})$ is contained in a halfspace by construction.\\
By the above, the flow $\mathcal{M}^\infty$ has dominant unstable mode. Thus, by \cite[Theorem 6.7 and Theorem 6.9]{CHH_wing} associated to $\mathcal{M}^\infty$ there is a nonvanishing fine bubble-sheet vector $(a_{1}^\infty, a_{2}^\infty)$. This contradicts the fact that the fine-bubble sheet vector of $\mathcal{M}^i$ is obtained from the fine bubble-sheet vector $(a_1,a_2)$ of $\mathcal{M}$ by multiplying by  $Z(X^\pm_i)^{-1}\to 0$, and thus proves the claim.
\end{proof}

Continuing the proof of the theorem, consider the sequence
\begin{equation}
    \mathcal{M}^{\pm,i}=\mathcal{M}-(p^{\pm}_i,t_i) \, ,
\end{equation}
which is obtained by shifting in space-time without rescaling. By \cite[Theorem 1.14]{HaslhoferKleiner_meanconvex} we can pass to subsequential limits $\mathcal{M}^\pm$, which are ancient noncollapsed flows that are weakly convex and smooth until they becomes extinct. By Claim \ref{cyl_scale_bd} (bubble-sheet scale) and since $|x_1(p_{t_i}^\pm)| \to \infty$, we infer that $\mathcal{M}^\pm$ is noncompact and has a bubble-sheet tangent at $-\infty$. Moreover, the argument from the proof of Claim \ref{cyl_scale_bd} also yields that $\mathcal{M}^\pm$ has dominant unstable mode, with the same fine bubble-sheet vector $(a_1,a_2)$ as $\mathcal{M}$.
Now, by \cite[Theorem 1.10]{CHH_wing} the flows $\mathcal{M}^\pm$ must be $\mathbb{R}\times$2d-bowl. However, observe that by our choice of points in \eqref{choice_min_max}, the flow $\mathcal{M}^-$ translates in positive $x_1$-direction, while the flow $\mathcal{M}^+$ translates in negative $x_1$-direction. In particular, the fine bubble-sheet vector of $\mathcal{M}^-$ points in positive $x_1$-direction, while the fine bubble-sheet vector of $\mathcal{M}^+$ points in negative $x_1$-direction. This contradicts the fact that $\mathcal{M}^-$ and $\mathcal{M}^+$ have the same fine-bubble sheet vector, and thus concludes the proof of the theorem.
\end{proof}

\bigskip

\section{The half-degenerate case}\label{sec_half_deg}

In this short section, we prove Theorem \ref{thm_halfdeg} (half-degenerate case), which we restate here for convenience of the reader.

\begin{theorem}[half-degenerate case]
Let $M_t$ be an ancient noncollapsed mean curvature flow in $\mathbb{R}^{4}$ whose tangent flow at $-\infty$ is given by \eqref{bubble-sheet_tangent_intro}.
If $\mathrm{rk}(Q)=1$, and  if $M_t$ is either splits off a line or is selfsimilarly translating, then $M_t$ is either $\mathbb{R}\times$2d-oval or belongs to the one-parameter family of 3d oval-bowls constructed by Hoffman-Ilmanen-Martin-White, respectively.
\end{theorem}

\begin{proof}
Consider first the case that our flow splits off a line, i.e. that $M_t=\mathbb{R}\times N_t$ is the product of a line and a two-dimensional ancient noncollapsed flow $N_t$. Then, by the work of Brendle-Choi \cite{BC1} and Angenent-Daskalopoulos-Sesum \cite{ADS1,ADS2} the flow $N_t$ must be either flat $\mathbb{R}^2$, a round shrinking $S^2$, a round shrinking $ \mathbb{R}\times S^1$, a 2d-bowl, or an ancient 2d-oval. However, $\mathbb{R}^2$ and $S^2$ are excluded by the assumption that $M_t$ has a bubble-sheet tangent at $-\infty$. Moreover, if $N_t$ was $\mathbb{R}\times S^1$ or a 2d-bowl, then the renormalized flow $\bar{M}_\tau$ would converge exponentially fast to the bubble-sheet, contradicting the assumption that $\mathrm{rk}(Q)=1$. Hence, the only possibility is that $N_t$ is a 2d-oval. Then, by \cite[Theorem 1.6]{ADS1} the profile function $v(y,\tau)$ of the renormalized flow $\bar{N}_\tau$ satisfies
\begin{equation}
v(y,\tau)=-\frac{y^2-2}{\sqrt{8}\tau}+o(|\tau|^{-1})\, .
\end{equation}
Consequently, the fine bubble-sheet function $u(y_1,y_2,\tau)$ of $\bar{M}_\tau=\mathbb{R}\times \bar{N}_\tau$ satisfies
\begin{equation}\label{expansion_inward_bend}
u(y_1,y_2,\tau)=-\frac{y_2^2-2}{\sqrt{8}\tau}+o(|\tau|^{-1})\, ,
\end{equation}
and hence the expansion from Theorem \ref{spectral theorem} (bubble-sheet quantization) indeed holds with
\begin{equation}\label{eq_mat_q_half}
Q=
\begin{pmatrix}
0 & 0\\
0 & -1/\sqrt{8}
\end{pmatrix}\, .
\end{equation}
Consider now the case that our flow is selfsimilarly translating. Then, by the recent classification by Choi, Hershkovits and the second author \cite{CHH_translator} it is either $\mathbb{R}\times$2d-bowl, or a 3d round bowl, or belongs to the one-parameter family of 3d oval-bowls constructed by Hoffman-Ilmanen-Martin-White. The case of the 3d round bowl is excluded by the assumption that the tangent flow at $-\infty$ is a bubble-sheet, and the case $\mathbb{R}\times$2d-bowl is excluded by the assumption $\mathrm{rk}(Q)=1$. Finally, by \cite[Theorem 3.6]{CHH_translator} for the 3d oval-bowls we indeed have the same asymptotics as in \eqref{expansion_inward_bend}, and hence the expansion from Theorem \ref{spectral theorem} (bubble-sheet quantization)  holds again with the matrix $Q$ from \eqref{eq_mat_q_half}. This finishes the proof of the theorem.
\end{proof}

\bigskip

\section{The non-degenerate case}\label{sec_non_deg}

In this final section, we prove Theorem \ref{Rk=2} (non-degenerate case), which we restate here for convenience of the reader:

\begin{theorem}[non-degenerate case]\label{Rk=2_rest}Let $M_t$ be an ancient noncollapsed mean curvature flow in $\mathbb{R}^{4}$ whose tangent flow at $-\infty$ is given by \eqref{bubble-sheet_tangent_intro}.
If $\mathrm{rk}(Q)=2$, then $M_t$ is compact and $\mathrm{SO}(2)$-symmetric and satisfies the following sharp asymptotics:
\begin{itemize}
    \item Parabolic region: The bubble-sheet function $u$ for $\tau\to -\infty$ satisfies
    \begin{equation*}
     u(y_1,y_2,\theta,\tau)=\frac{y_1^2+y_2^2-4}{\sqrt{8}\tau} +o(|\tau|^{-1})
    \end{equation*}
        uniformly for $|(y_1,y_2)|\leq R$.
    \item Intermediate region: We have
\begin{equation*}
    \lim_{\tau\rightarrow -\infty}u(|\tau|^{\frac{1}{2}}z_1,|\tau|^{\frac{1}{2}}z_2,\theta, \tau)+\sqrt{2}=\sqrt{2-(z_1^2+z_2^2)}
\end{equation*}
    uniformly on every compact subset of $\{ z_1^2+z_2^2<\sqrt{2}\}$.
    \item Tip region:  Set $\lambda(s)=\sqrt{|s|^{-1}\log |s|}$, and given any angle $\phi$ let $p_{s}\in M_s$ be the point that maximizes $\langle p, \cos(\phi) e_1 + \sin(\phi) e_2\rangle$ among all $p\in M_s$. Then, as $s\to -\infty$ the rescaled flows 
    \begin{equation*}
        {\widetilde{M}}^{s}_{t}=\lambda(s)\cdot(M_{s+\lambda(s)^{-2}t}-p_{s})
    \end{equation*}
    converge to $\mathbb{R}\times N_{t}$, where $N_{t}$ is the 2d-bowl in $\mathbb{R}^{3}$ with speed $1/\sqrt{2}$.
\end{itemize}
\end{theorem}

We will first establish the sharp asymptotics, which in particular will imply compactness, and afterwards will prove the circular symmetry.\\

As before, we work with the truncated bubble-sheet function
\begin{equation}
\hat{u}(y_1,y_2,\theta,\tau)=u(y_1,y_2,\theta,\tau)\chi\left(\frac{|(y_1,y_2)|}{|\tau|^\gamma} \right)\, ,
\end{equation}
where the exponent $\gamma>0$ is from Proposition \ref{radius_lower_bound} (graphical radius).

\begin{proposition}[parabolic region]\label{prop_parabolic}
The truncated bubble-sheet function $\hat{u}$ satisfies
\begin{equation}\label{inw_quad11}
\lim_{\tau\to -\infty} \left\|  \tau \hat{u}(y_1,y_2,\theta,\tau) - \frac{y_1^2+y_2^2-4}{\sqrt{8}} \right\|_{\mathcal{H}}=0\, .
\end{equation}
Moreover, there exist $\tau_\ast>-\infty$ and an increasing function $\delta:(-\infty,\tau_\ast)\to (0,1/100)$ with $\lim_{\tau\to -\infty}\delta(\tau)=0$ such that for $\tau\leq \tau_\ast$ we have
    \begin{equation}\label{inw_quad22}
\sup_{|(y_1,y_2)|\leq \delta(\tau)^{-1}} \left|    u(y_1,y_2,\theta,\tau) - \frac{y_1^2+y_2^2-4}{\sqrt{8}\tau} \right|\leq \frac{\delta(\tau)}{|\tau|}\, .
    \end{equation}
    In particular, the hypersurfaces $M_t$ are compact.

\end{proposition}
\begin{proof}
By the assumption $\mathrm{rk}(Q)=2$ we have
\begin{equation}\label{eq_mat_q}
Q=
\begin{pmatrix}
-1/\sqrt{8} & 0\\
0 & -1/\sqrt{8}
\end{pmatrix}\, ,
\end{equation}
hence
\begin{equation}
y^\top Q y - 2\mathrm{tr}(Q) = -\frac{y_1^2+y_2^2-4}{\sqrt{8}}\, .
\end{equation}
Applying Theorem \ref{spectral theorem_restated}  (bubble-sheet quantization) this proves \eqref{inw_quad11}. Together with standard parabolic estimates this yields \eqref{inw_quad22}. Finally, together with convexity this implies that the hypersurfaces $M_t$ are compact.
\end{proof}

Next, to capture the intermediate region, we consider the function
\begin{equation}
\bar{u}(z_1,z_2,\theta, \tau)=\sqrt{2}+u(|\tau|^{\frac{1}{2}}z_1,|\tau|^{\frac{1}{2}}z_2,\theta, \tau)\, .
\end{equation}

\begin{proposition}[intermediate region lower bound]\label{prop_inter_lower}
For any compact subset $K\subset \{ z_1^2+z_2^2<2\}$ we have
\begin{equation}
    \liminf_{\tau\rightarrow-\infty} \inf_{(z_1,z_2)\in K} \left(\bar{u}(z_1,z_2,\theta, \tau)-\sqrt{2-(z_1^2+z_2^2)}\right)\geq 0.
\end{equation}
\end{proposition}

\begin{proof}We will generalize the argument from our prior paper \cite[proof of Proposition 2.8]{DH_ovals} to the setting without symmetry assumptions.\\
By \cite[Lemma 4.4]{ADS1} there exists an increasing positive function $M(a)$ with $\lim_{a\to \infty}M(a)= \infty$, such that the profile function $u_a$ of the ADS-barrier $\Sigma_a$ defined in \eqref{ads_shrinker} for $0\leq r\leq M(a)$ satisfies
\begin{equation}
    u_{a}(r)\leq \sqrt{2}-\frac{r^2-3}{\sqrt{2}a^2}\, .
\end{equation}
We fix $\tau_\ast$ negative enough, and for $\tau\leq\tau_\ast$ set
\begin{equation}
    L(\tau)=\min\{\delta(\tau)^{-1}, M(|\tau|^{\frac{1}{2}}), |\tau|^{\frac{1}{2}-\frac{1}{100}}\} \, ,\quad \hat{a}(\tau)=\sqrt{\frac{2|\tau|}{1+L(\tau)^{-1}}}\, ,   
\end{equation}
where $\delta(\tau)$ is the function from Proposition \ref{prop_parabolic} (parabolic region).
Then, for $\hat{\tau}\leq\tau_\ast$, we get
\begin{align}\label{b2}
    u_{\hat{a}(\hat{\tau})}\left(L(\hat{\tau})+3L(\hat{\tau}\right)^{-1})
    \leq \sqrt{2} -\frac{L(\hat{\tau})^2}{\sqrt{8}|\hat{\tau}|}\, .
\end{align}
On the other hand, by Proposition \ref{prop_parabolic} (parabolic region) we have
\begin{equation}\label{b1}
       \sqrt{2}+ u(y_1,y_2,\theta, \tau)\geq \sqrt{2}-\frac{L(\hat{\tau})^2}{\sqrt{8}|\hat{\tau}|}\, ,
    \end{equation}
whenever $|(y_1,y_2)|=L(\hat{\tau})$ and $\tau\leq \hat{\tau}$. 
Hence, if we consider the shifted and rotated hypersurfaces
\begin{equation}\label{rotated_barrier_rest}
\Gamma_a^\eta=\{(r\cos\theta ,r\sin\theta,y_3,y_4)\in  \mathbb{R}^4:\theta\in [0,2\pi), (r-\eta,y_3,y_4) \in {\Sigma}_a \},
\end{equation}
with $\eta=\hat{\eta}(\hat{\tau})=3L(\hat{\tau})^{-1}$ and $a=\hat{a}(\hat{\tau})$ as above, then applying the inner barrier principle from \cite[Proposition 3.1]{CHH_translator}  we infer that
\begin{equation}
   \sqrt{2}+u(y_1,y_2,\theta,\tau) \geq    u_{\hat{a}(\hat{\tau})}\left(|(y_1,y_2)|+\hat{\eta}(\hat{\tau})\right) \, ,
\end{equation}
whenever  $ |(y_1,y_2)|\geq L(\hat{\tau})$ and  $\tau\leq \hat{\tau}$. Moreover, by \cite[Lemma 4.3]{ADS1}, we have
\begin{equation}
u_a(r)=\sqrt{2-\frac{2r^2}{a^2}}+o(1)
\end{equation}
uniformly in $r$ as $a\to \infty$. Together with convexity we conclude that
\begin{equation}
    \liminf_{\tau\rightarrow-\infty} \inf_{(z_1,z_2)\in K} \left(\bar{u}(z_1,z_2,\theta, \tau)-\sqrt{2-(z_1^2+z_2^2)}\right)\geq 0.
\end{equation}
This proves the proposition.
\end{proof}

Observe that by the lower bound, for any angle $\phi$ we have
 \begin{equation}
 \max_{p\in M_t}\langle p, \cos(\phi) e_1 + \sin(\phi) e_2\rangle \geq  \sqrt{(2-o(1))|t|\log|t|}\, .
\end{equation}
In particular, the bubble-sheet function $u(y_1,y_2,\theta,\tau)$ is well-defined whenever $|(y_1,y_2)|\leq \sqrt{(2-o(1))|\tau|}$.
To proceed, we prove almost symmetry away from these tip regions:

\begin{lemma}[almost symmetry]\label{lemma_almost_symm_further}
For every $\delta>0$, there exist constants $\eta>0$ and $\tau_\ast>-\infty$, such  that for all $\tau\leq\tau_\ast$ we have
\begin{equation}
 \sup_{|(y_1,y_2)|\leq \sqrt{(2-\delta)|\tau|}}    |u_{\theta}(y_1,y_2,\theta,\tau)|\leq e^{-\eta |\tau|^{1/2}}\, .
\end{equation}
\end{lemma}

\begin{proof} We first claim that given any $\eps_0>0$ we can find $T_0>-\infty$ so that all space-time points $X=(p,t)\in \mathcal{M}$ with $|p|\leq \sqrt{(2-\delta/2)|t|\log|t|}$ and $t\leq T_0$ lie on the center of an $\eps_0$-bubble-sheet. Indeed, if there were a sequence $t_i\to -\infty$ and points $p_i\in M_{t_i}$ that do not lie on the center of an $\eps_0$-bubble-sheet then by the global convergence theorem \cite[Theorem 1.12]{HaslhoferKleiner_meanconvex} after passing to a subsequence
\begin{equation}
\widetilde{M}^i_t:=H(p_i,t_i)\cdot(M_{t_i+H(p_i,t_i)^{-2}t}-p_{i})
\end{equation}
would converge to a limit that splits off two lines, and hence by \cite[Lemma 3.14]{HaslhoferKleiner_meanconvex} is a round shrinking $\mathbb{R}^2\times S^1$. This is a contradiction and thus establishes the claim. Using the property that every point under consideration is $\eps_0$-close to a bubble-sheet, and arguing similarly as in the proof of Proposition \ref{utheta} (almost circular symmetry) the assertion follows.
\end{proof}

We can now establish a matching upper bound:

\begin{proposition}[intermediate region upper bound]\label{prop_inter_upper}
For any compact subset $K\subset \{ z_1^2+z_2^2<2\}$ we have
\begin{equation}
    \limsup_{\tau\rightarrow-\infty} \sup_{(z_1,z_2)\in K} \left(\bar{u}(z_1,z_2,\theta, \tau)-\sqrt{2-(z_1^2+z_2^2)}\right)\leq 0.
\end{equation}
\end{proposition}

\begin{proof}
We will generalize the argument from our prior paper \cite[proof of Proposition 2.8]{DH_ovals} to the setting without symmetry assumptions.\\
By Proposition \ref{MCF_graphical_equation} (evolution over cylinder), convexity, and Lemma \ref{lemma_almost_symm_further} (almost symmetry) the graphical function $v=\sqrt{2}+u$ satisfies
\begin{align}\label{v_evolution_inequ}
    v_\tau\leq 
     -\frac{1}{v}+\frac{1}{2}\left(v- y_\alpha \partial_\alpha v\right)+ C e^{-\eta |\tau|^{1/2}} \, 
\end{align}
for $\tau\ll0$ and $|(y_1,y_2)|\leq \sqrt{(2-\delta)|\tau|}$, where $C=C(\delta)<\infty$. Now, given any angles $\phi$ and $\theta$, considering the function
\begin{equation}
w(\rho,\tau):=v(\rho\cos\phi,\rho\sin\phi,\theta,\tau)^2-2\, ,
\end{equation}
we infer that
\begin{equation}
   w_\tau\leq w-\tfrac12 \rho w_\rho + Ce^{-\eta|\tau|^{1/2}}\, . 
\end{equation}
Hence, for every $\rho_{0}>0$ we have
\begin{equation}
    \frac{d}{d\tau}(e^{-\tau}w(\rho_{0} e^{\frac{\tau}{2}}, \tau))\leq Ce^{-\tau-\eta|\tau|^{1/2}}
\end{equation}
for $\tau\ll 0$. Integrating this inequality, we obtain for every $\lambda\in (0, 1]$ that
\begin{equation}\label{ODEv}
    w(\rho, \tau)\leq \lambda^{-2}w(\lambda \rho, \tau+2\log\lambda)+o(|\tau|^{-1}).
\end{equation}
On the other hand, by Proposition \ref{prop_parabolic} (parabolic region), given any $A<\infty$, the inequality
 \begin{equation}\label{limitv}
        w(\rho, \tau)\leq |\tau|^{-1}({4-\rho^2})+o(|\tau|^{-1})
    \end{equation}
    holds for $\rho\leq A$.
Thus, for $\rho\geq A$ we obtain
\begin{equation}
    w(\rho,\tau)\leq -\frac{(1-2A^{-2})\rho^2}{|\tau|+2 \log(\rho/A)}+o(|\tau|^{-1}).
\end{equation}
This implies the assertion.
\end{proof}

To analyze the tip region we define the rescaling factor
\begin{equation}
\lambda(s)=\sqrt{|s|^{-1}\log |s|}\, ,
\end{equation}
and given any angle $\phi$ consider the point $p_{s}\in M_s$ where
\begin{equation}
 \max_{p\in M_s}\langle p, \cos(\phi) e_1 + \sin(\phi) e_2\rangle
\end{equation}
is attained (note that $p_s$ is unique by strict convexity). Then, we have:
 
\begin{proposition}[tip region]\label{prop_tip}
 For $s\to -\infty$ the rescaled flows 
    \begin{equation}
        {\widetilde{M}}^{s}_{t}=\lambda(s)\cdot(M_{s+\lambda(s)^{-2}t}-p_{s})
    \end{equation}
    converge to $L\times N_{t}$, where $L$ is the line spanned by $-\sin(\phi)e_1+\cos(\phi)e_2$ and $N_{t}$ is the 2d-bowl with speed $1/\sqrt{2}$ moving in direction $-(\cos(\phi) e_1 + \sin(\phi) e_2)$.
\end{proposition}

\begin{proof}
After rotating coordinates we can assume without loss of generality that $\phi=0$. Consider the distance of $p_t$ from the origin, namely
\begin{equation}
d(t):=|p_t|\, .
\end{equation}
Using the above propositions about the intermediate region and convexity we see that
\begin{equation}\label{diameter_asymptotics}
  d(t)=\sqrt{2|t| \log|t|}(1+o(1)).
\end{equation}
Moreover, by Hamilton's Harnack inequality \cite{Hamilton_Harnack} we have
\begin{equation}
    \frac{d}{dt}H(p_t)\geq 0.
\end{equation}
Together with the mean curvature flow equation this yields
\begin{equation}\label{tip curvature}
  \frac{H(p_t)}{\sqrt{|t|^{-1}\log|t|}}=\frac{1}{\sqrt{2}}+o(1).
\end{equation}
Next, consider the points $q_t^\pm \in M_t$ where 
\begin{equation}
 \max_{q\in M_t}\langle q, \pm e_2\rangle
\end{equation}
is attained. Then, by convexity the domain $K_t$ enclosed by $M_t$ contains the triangle with vertices $p_t,q_t^\pm$, and using  again the above propositions about the intermediate region we see that
\begin{equation}\label{eq_angle_cond}
\angle q_t^- p_t q_t^+\geq \frac{\pi}{4}\, 
\end{equation}
for $t\ll 0$. Now, given any sequence $s_i\to -\infty$, by the global convergence theorem \cite[Theorem 1.12]{HaslhoferKleiner_meanconvex} the sequence
\begin{equation}
\widetilde{K}^{s_i}_t:=\lambda(s_i)\cdot(M_{s_i+\lambda(s_i)^{-2}t}-p_{s_i})
\end{equation}
 converges subsequentially to a limit $K^\infty_t$. Observe that $K^\infty_t$ is a noncompact ancient noncollapsed flow that thanks to \eqref{eq_angle_cond} contains a wedge. Hence, by \cite[Theorem 1.4]{CHH_wing} it follows that $M^{\infty}_t=\partial K^\infty_t$ is the product of a line and a 2d-bowl. By construction, the line is in $e_2$-direction, and the 2d-bowl translates in negative $e_1$-direction with speed $1/\sqrt{2}$. Finally, by uniqueness of the limit, the subsequential convergence entails full convergence. This concludes the proof of the proposition.
\end{proof}

Finally, we can prove circular symmetry:

\begin{proposition}[circular symmetry]\label{circular symmetry2}
If $\mathcal{M}=\{M_t\}$ is an ancient noncollapsed mean curvature flow in $\mathbb{R}^{4}$ whose tangent flow at $-\infty$ is given by \eqref{bubble-sheet_tangent_intro} and such that
$\mathrm{rk}(Q)=2$, then $M_t$ is $\mathrm{SO}(2)$-symmetric.
\end{proposition}

Recall that by Zhu's bubble-sheet improvement \cite[Theorem 3.7]{Zhu} there exist constants $\eps_0>0$ and $L_0<\infty$ with the following significance. Given $\varepsilon\leq\varepsilon_{0}$ and $X\in \mathcal{M}$, if every $X'\in \mathcal{M}\cap P(X, L_0/H(X))$ is $\varepsilon_{0}$-close to a bubble-sheet $\mathbb{R}^2\times S^1$ and $\varepsilon$-symmetric,  then $X$ is $\eps/2$-symmetric. Furthermore, by Zhu's cap improvement \cite[Theorem 3.8]{Zhu} there constants $\eps_1>0$ and $L_1<\infty$ with the following significance. If  $\mathcal{M}\cap P(X, L_1/H(X))$ is $\varepsilon_1$-close to a piece of $\mathrm{Bowl}_2\times\mathbb{R}$, and for some $\eps\leq\eps_1$ every point in $X'\in \mathcal{M}\cap P(X, L_1/H(X))$ is $\eps$-symmetric, then $X$ is $\eps/2$-symmetric.\\

To apply this we need canonical neighborhoods:

\begin{lemma}[canonical neighborhoods]\label{can_nbd}
For every $\varepsilon'>0$, there exists a $t_\ast>-\infty$, such that for $t\leq t_\ast$ every $(p,t)\in\mathcal{M}$ is $\eps'$-close either to a round shrinking bubble-sheet $\mathbb{R}^2\times S^1$ or to a piece of a translating $\mathrm{Bowl}_2\times\mathbb{R}$.
\end{lemma}

\begin{proof} Suppose towards a contradiction for some $t_i\to -\infty$ there are $(p_i,t_i)\in \mathcal{M}$ that are neither $\eps'$-close to a round shrinking bubble-sheet $\mathbb{R}^2\times S^1$ nor to a piece of a translating $\mathrm{Bowl}_2\times\mathbb{R}$. Let $T_t$ be the set of tip points at time $t$, namely $T_t=\cup_\phi \{p_t(\phi)\}$, where for any angle $\phi$ we denote by $p_t(\phi)\in M_t$ the point at which $\max_{p\in M_t}\langle p,\cos(\phi) e_1 + \sin(\phi) e_2\rangle$ is attained. After passing to a subsequence we can assume that $H(p_i,t_i)\textrm{dist}(p_i,T_{t_i})$ either converges to infinity or to a finite constant. However, by the global convergence theorem \cite[Theorem 1.12]{HaslhoferKleiner_meanconvex} after passing to a further subsequence
\begin{equation}
\widetilde{M}^i_t:=H(p_i,t_i)\cdot(M_{t_i+H(p_i,t_i)^{-2}t}-p_{i})
\end{equation}
converges to a limit, which in the first case splits off two lines and hence must be a round shrinking bubble-sheet $\mathbb{R}^2\times S^1$, and in the second case thanks to Proposition \ref{prop_tip} (tip region) is a piece of a translating $\mathrm{Bowl}_2\times\mathbb{R}$. This gives the desired contradiction, and thus proves the lemma.
\end{proof}

Having established canonical neighborhoods, we can now conclude the argument similarly as in \cite[proof of Theorem 1.5]{ADS2}: 

\begin{proof}[Proof of Proposition \ref{circular symmetry2}]
Fix $\eps'\ll \eps:=\min(\eps_1,\eps_2)$. By Lemma \ref{can_nbd} (canonical neighborhoods) there exists $t_\ast>-\infty$, such that for $t\leq t_\ast$ every $(p,t)\in\mathcal{M}$ is $\eps'$-close either to a bubble-sheet $\mathbb{R}^2\times S^1$ or to a piece of $\mathrm{Bowl}_2\times\mathbb{R}$. In particular, by our choice of constants any such point is $\eps$-symmetric. Hence, by Zhu's bubble-sheet and cap improvement \cite[Theorem 3.7 and Theorem 3.8]{Zhu} for $t\leq t_\ast$ every $(p,t)\in\mathcal{M}$  is $\eps/2$-symmetric. Iterating this, we infer that given any positive integer $j$, for $t\leq t_\ast$ every $(p,t)\in\mathcal{M}$  is $\eps/2^j$-symmetric. Since $j$ is arbitrary, this implies that $M_t$ is $\mathrm{SO}(2)$-symmetric for $t\leq t_\ast$. Finally, by uniqueness of closed smooth solutions of the mean curvature flow the $\mathrm{SO}(2)$-symmetry is preserved also forwards in time. This concludes the proof of the proposition.
\end{proof}

\bigskip

Theorem \ref{Rk=2} (non-degenerate case) now follows by combining   the sharp asymptotics from above and Proposition \ref{circular symmetry2} (circular symmetry).

\bigskip

 \begin{appendix}
 \section{Graphical evolution over cylinders}\label{ApA}
Suppose $\bar{M}_\tau\subset\mathbb{R}^{n+1}$ moves by renormalized mean curvature flow
\begin{equation}
\partial_\tau y = \vec{H}+\frac{y^\perp}{2},
\end{equation}
and suppose $\bar{M}_\tau$ can be locally parametrized over $\mathbb{R}^{k}\times S^{n-k}$ via
\begin{equation}
(z,\omega)\mapsto (z,v(z,\omega,\tau)\omega),
\end{equation}
where $z\in U$ for some open set $U\subseteq \mathbb{R}^k$, and $\omega\in S^{n-k}$. We write $\{\partial_{\alpha}\}_{\alpha=1,\ldots,k}$ for derivatives along the $\mathbb{R}^k$-factor, and $\{\nabla_i\}_{i=k+1,\ldots,n}$ for (covariant) derivatives along the $S^{n-k}$-factor. The goal of this appendix is to prove:

\begin{proposition}[evolution over cylinder]\label{MCF_graphical_equation}
The graphical function $v$ evolves by
\begin{align}\label{v_evolution}
    \partial_{\tau}v=& \frac{A_{\alpha\beta}\partial_{\alpha}\partial_{\beta}v+B_{ij}\nabla_{i}\nabla_{j}v-2\partial_{\alpha} v  \nabla_i v \nabla_i\partial_{\alpha} v -v^{-1}|\nabla v|^2}{(1+|\partial v|^2)v^2+|\nabla v|^2}\\
    & -\frac{n-k}{v}+\frac{1}{2}\left(v- z_\alpha \partial_\alpha v\right),\nonumber
\end{align}
where 
\begin{equation}\label{A}
    A_{\alpha\beta}=[(1+|\partial v|^2)v^2+|\nabla v|^2]\delta_{\alpha\beta}-v^{2}\partial_{\alpha}v \partial_{\beta}v,
\end{equation}
and
\begin{equation}\label{B}
    B_{ij}=(1+|\partial v|^2+v^{-2}|\nabla v|^2)\delta_{ij}-v^{-2}\nabla_{i}v\nabla_{j}v.
\end{equation}
\end{proposition}
\begin{proof} We will generalize the computation from  \cite[Appendix A]{GKS} from necks to general cylinders.
Considering the parametrization
\begin{equation}
    F(z, \omega,\tau)=(z, v(z, \omega,\tau)\omega),
\end{equation}
and working in an orthonormal basis $e_i\in T_\omega S^{n-k}$ we compute
\begin{equation}
    D_{\alpha}F=e_{\alpha}+\partial_{\alpha}v\, \omega,\qquad D_{i}F=ve_{i} + \nabla_{i}v\, \omega\, .
\end{equation}
Thus, the outwards unit normal is
 \begin{equation}\label{nunormal}
    \nu=\frac{N}{|N|}\, , \qquad\textrm{where } N=v\, \omega - \nabla v   - v \partial v \, .
\end{equation}
Since $\bar{M}_\tau$ moves by renormalized mean curvature flow, the graph function $v$ evolves by
\begin{equation}
\partial_\tau v = \frac{1}{\langle \omega ,\nu\rangle} \left\langle \vec{H}+\frac{1}{2} F,\nu\right\rangle\, .
\end{equation}
One easily sees that
\begin{equation}
|N|=\sqrt{(1+|\partial v|^2)v^2+|\nabla v|^2}\, , \qquad \frac{\langle  F,\nu\rangle}{\langle \omega ,\nu\rangle}= v -z_\alpha \partial_\alpha v\, .
\end{equation}
Hence, our main task is to compute the mean curvature $H= -\langle \vec{H},\nu\rangle$.
To this end, observe first that the components of the metric  are
\begin{align}
    g_{\alpha\beta}=\delta_{\alpha\beta}+\partial_{\alpha}v \partial_{\beta}v\, ,\quad     g_{\alpha i}=\partial_{\alpha}v \nabla_{i} v\, ,
\end{align}
and
\begin{align}    
        g_{ij}=v^2\delta_{ij}+\nabla_{i}v \nabla_{j}v\, .
\end{align}
Next, we compute
\begin{equation}
D_{\alpha}D_{\beta}F=\partial_{\alpha}\partial_\beta v\,  \omega\, , \quad D_\alpha D_i F = \nabla_i \partial_\alpha v \, \omega +\partial_\alpha v\, e_i\, , 
\end{equation}
and
\begin{equation}
\quad D_iD_j F =\nabla_i v\,  e_j+\nabla_j v \, e_i-v\delta_{ij}\omega+ \nabla_i\nabla_j v\, \omega\, .
\end{equation}
Thus, the components of the second fundamental form are
\begin{equation}
h_{\alpha\beta}= -\frac{v \partial_\alpha\partial_\beta v}{|N|}\, ,\quad   h_{\alpha i}=\frac{\partial_{\alpha}v \nabla_{i} v-v\nabla_{i}\partial_{\alpha}v}{|N|}\, ,
\end{equation}
and 
\begin{equation}
 \quad     h_{ij}=\frac{v^2\delta_{ij}+ 2\nabla_i v\nabla_j v-v\nabla_i\nabla_j v}{|N|}  \, .
\end{equation}
Now, by symmetry we may assume that $\partial_{k}v$ and $\nabla_{k+1}v$ are the only nonvanishing first derivates at the point in consideration. Then, we can identify the metric with the block matrix
\begin{equation}
    g=
    \left (\begin{array}{cccc}
       I_{k-1} & 0 & 0 & 0\\
       0 &  1+|\partial v|^{2}  & \partial_{k}v \nabla_{k+1}v& 0\\
       0 &   \partial_{k}v \nabla_{k+1}v & v^{2}+|\nabla v|^{2}&0\\
       0&0&0& v^2I_{n-k-1}
    \end{array}\right)\, .
\end{equation}
Using this we infer that the mean curvature equals
\begin{align}
H = -\frac{v\sum_{\alpha\neq k} \partial_\alpha^2v}{|N|} +\mathrm{tr}(PQ)
+\frac{ (n-k-1)v^2- v \sum_{i\neq k+1} \nabla^2_i v }{v^2|N|}\, ,
\end{align}
where
\begin{equation}
    P=
   \frac{1}{|N|^2} \left (\begin{array}{cc}
      v^{2}+|\nabla v|^{2}    & - \partial_{k}v \nabla_{k+1}v\\
   -  \partial_{k}v \nabla_{k+1}v & 1+|\partial v|^{2}
    \end{array}\right)\, ,
\end{equation}
and
\begin{equation}
    Q= \frac{1}{|N|}\left (\begin{array}{cc}
      - v \partial_k^2v   &  \partial_k v\nabla_{k+1}v -  v \nabla_{k+1}\partial_k v\\
 \partial_k v\nabla_{k+1}v -  v \nabla_{k+1}\partial_k v & v^2 + 2|\nabla v|^2-v\nabla_{k+1}^2 v
    \end{array}\right)\, .
\end{equation}
Computing $\mathrm{tr}(PQ)$ and recasting all the terms in coordinate-independent notation we conclude that
\begin{align}\label{mean_curvature}
   - \frac{|N|}{v}H=\frac{A_{\alpha\beta}\partial_{\alpha}\partial_{\beta}v+B_{ij}\nabla_{i}\nabla_{j}v-2\partial_{\alpha} v \nabla_i v \nabla_i \partial_{\alpha}v-v^{-1}|\nabla v|^{2}}{|N|^{2}}
    -\frac{n-k}{v}\, ,\nonumber
\end{align}
where $A_{\alpha\beta}, B_{ij}$ are defined in \eqref{A} and \eqref{B}.
This yields the assertion.
\end{proof}

\end{appendix}

\bigskip

\bibliography{hearing_shape}

\bibliographystyle{alpha}

%\vspace{10mm}

{\sc Wenkui du, Department of Mathematics, University of Toronto,  40 St George Street, Toronto, ON M5S 2E4, Canada}

{\sc Robert Haslhofer, Department of Mathematics, University of Toronto,  40 St George Street, Toronto, ON M5S 2E4, Canada}

\emph{E-mail:} wenkui.du@mail.utoronto.ca, roberth@math.toronto.edu

\end{document}